\documentclass[11pt]{amsart}
\usepackage[a4paper,margin=2.4cm]{geometry}

\usepackage{setspace}
\usepackage{amssymb}
\usepackage{amsmath}
\usepackage{amsthm}

\usepackage[shortlabels]{enumitem}
\usepackage{tikz}
\usepackage{arydshln}

\newtheorem{theorem}{Theorem}[section]
\newtheorem*{theorem*}{Theorem}
\newtheorem{proposition}[theorem]{Proposition}
\newtheorem{lemma}[theorem]{Lemma}

\theoremstyle{definition}
\newtheorem{definition}[theorem]{Definition}

\newtheorem{remark}[theorem]{Remark}

\newtheoremstyle{named}{}{}{\itshape}{}{\bfseries}{.}{.5em}{\thmnote{#3}}
\theoremstyle{named}

\usepackage[pdftex,unicode]{hyperref}   % Must follow all other packages
\hypersetup{unicode}
\hypersetup{breaklinks=true}
\hypersetup{urlcolor=blue}

\newcommand{\F}{\mathbb F}

\renewcommand{\P}{\mathbb P}
\renewcommand{\S}{\mathbb S}
\newcommand{\f}    [1]{\mathbb{F}_{#1}}

\newcommand{\w}{\omega}
\newcommand{\U}{\underline}

\newcommand{\N}{\mathbb N}

\newcommand{\faruk}{\textcolor[rgb]{0.1,0.1,0.9}}%blue
\DeclareMathOperator{\GammaL}{\Gamma L}
\DeclareMathOperator{\PGammaL}{P\Gamma L}

\DeclareMathOperator{\diag}{diag}
\DeclareMathOperator{\PGL}{PGL}
\DeclareMathOperator{\QH}{QH}
\DeclareMathOperator{\SQH}{SQH}
\DeclareMathOperator{\PQH}{PQH}
\DeclareMathOperator{\PSQH}{PSQH}

\DeclareMathOperator{\End}{End}
\DeclareMathOperator{\GL}{GL}
\DeclareMathOperator{\Aut}{Aut}
\DeclareMathOperator{\Autt}{Autt}
\DeclareMathOperator{\Tr}{Tr}
\DeclareMathOperator{\Nm}{Nm}
\DeclareMathOperator{\id}{id}

\DeclareMathOperator{\im}{Im}

\DeclareMathOperator{\Fix}{Fix}

\newcommand{\PN}[1]{{\mathbb{P}^{#1}}}

\newcommand{\parskipsize}{0.4em}
\usepackage{enumitem}
\setlist{  
  listparindent=\parindent,
  parsep=\parskip,
}

\begin{document}

\title{Commutative Semifields from bijections of the Desarguesian plane}
\author{Faruk G\"{o}lo\u{g}lu}
\address{Charles University}
\curraddr{}
\email{faruk.gologlu@mff.cuni.cz}
\thanks{}
\author{Lukas K\"olsch}
\address{University of South Florida}
\curraddr{}
\email{lukas.koelsch.math@gmail.com}
\thanks{}
\subjclass[2020]{Primary XXXXX; Secondary XXXXX}
\date{}
\dedicatory{}

\begin{abstract}
The Menichetti-Kaplansky theorem [J. Algebra \textbf{47}(2) (1977)] 
states that a finite semifield 
that is three-dimensional over its center is either a field 
or a twisted field of Albert. This implies that a quadratic 
homogeneous bijection of $\PN 2(\F_q)$ is equivalent to 
a Dembowski-Ostrom monomial. In this paper, we give a large class 
of semiquadratic homogeneous bijections of $\PN 2(\F_q)$ that are 
inequivalent to Dembowski-Ostrom monomials. Using these bijections, 
we construct a large family of commutative semifields  that are 
non-isotopic to finite fields or twisted fields, which in turn give 
rise to a large family of non-Desarguesian commutative semifield planes. 
Semiquadratic homogeneous bijections of $\PN 1(\F_q)$ have been classified only 
recently by the first-named author [Finite Fields Appl. \textbf{81} (2022)] 
and Ding and Zieve [Proc. London Math. Soc. \textbf{127}(2) (2023)]
with the result that all such bijections are either equivalent 
to Dembowski-Ostrom monomials or degenerate. 
We demonstrate that this is not the 
case for $\PN 2(\F_q)$. 

\faruk{Updated version: Pre-semifield/semifield issues fixed, several typos/small mistakes corrected.}

\end{abstract}

\maketitle

\setcounter{tocdepth}{1}
\tableofcontents

%%%%%%%%%%%%%%%%%%%%%%%%%%%%%%%%%%%%%%%%%%%%%%%%%%%%%%%%%%%%%%%%%%%%%%

\section{Introduction}
Certain groups of permutations of the vector space $\F_q^n$ (resp. projective
space $\PN{n-1}(\F_q)$) over a finite field $\F_q$ are naturally important: 
$\GL(n,q)$ and $\GammaL(n,q)$ (resp. $\PGL(n,q)$ and $\PGammaL(n,q)$).
We recall the following definitions.

\begin{definition}
Let $q = p^m$ where $p$ is a prime and $\sigma \in \Aut(\F_q)$.
\begin{enumerate}
\item A $\sigma$-semilinear form 
$\ell : \F_q^n \to \F_q$ is a function that satisfies
\begin{align*}
\ell(c\U x)  &= c^\sigma \ell(\U x) \textrm{ and},\\         
\ell(\U x+\U y) &= \ell(\U x)+\ell(\U y)
\end{align*}
for all $c \in \F_q$ and $\U x,\U y \in \F_q^n$.
It is called a linear form if $\sigma = \id$.
\item A $\sigma$-sesquilinear form 
$b : \F_q^n \times \F_q^n \to \F_q$ is a function 
such that 
$\U x \mapsto b(\U x,\U y)$ is a linear form for all fixed $\U y$, and
$\U y \mapsto b(\U x,\U y)$ is a $\sigma$-semilinear form for all fixed $\U x$.
It is called a bilinear form if $\sigma = \id$.
\item A $\sigma$-semiquadratic form $s : \F_q^n \to \F_q$ is a function 
that satisfies 
\begin{align}
s(c\U x)  &= c^{\sigma+1} s(\U x) \textrm{ and},\label{eq:semi1}\\         
s(\U x+\U y)-s(\U x)-s(\U y) &= b(\U x,\U y) + b(\U y,\U x)\label{eq:semi2}
\end{align}
for all $c \in \F_q$ and $\U x,\U y \in \F_q^n$,
where $b$ is a $\sigma$-sesquilinear form. It is called a quadratic form if $\sigma = \id$.
\end{enumerate}
\end{definition}

\begin{remark}
	A generalization of quadratic forms over division rings that satisfy Eqs.~\eqref{eq:semi1} and~\eqref{eq:semi2} was studied under the name of \emph{pseudo-quadratic forms} by, for instance, Tits~\cite[Chapter 8.2]{tits1974buildings}. However, restrictions are always made such that the corresponding sesquilinear form is reflexive, which for pseudo-quadratic forms over fields immediately implies that the corresponding automorphism $\sigma$ satisfies $\sigma^2=\id$. We do not impose this restriction. 
\end{remark}

The group $\GammaL(n,q)$ (resp. $\GL(n,q)$) 
is the set of bijective semilinear (resp. linear) maps 
$L : \F_q^n \to \F_q^n$ with $L= (L_1,\ldots,L_n)$ 
where $L_i$ are $\sigma$-semilinear (resp. linear) forms 
where $\sigma \in \Aut(\F_q)$ (resp. $\sigma = \id$) is fixed 
for a given $L$. Thus the following definition is natural.

\begin{definition}
The set $\SQH(n,q)$ (resp. $\QH(n,q)$) is the set of
semiquadratic (resp. quadratic) maps
$S : \F_q^n \to \F_q^n$ with $S= (S_1,\ldots,S_n)$ 
where $S_i$ are $\sigma$-semiquadratic (resp. quadratic) forms 
where $\sigma \in \Aut(\F_q)$ (resp. $\sigma = \id$) is fixed 
for a given $S$. 
\end{definition}

\begin{remark}[Matrices and cubes]
It is easy to show that these maps can be written as the following 
matrix operations. For a vector $\U x = (x_1,\ldots,x_n)$,
define $\U x^\sigma = (x_1^\sigma,\ldots,x_n^\sigma)$ and let $\U x^T$ denote
the transpose as usual.
\begin{enumerate}
\item A semilinear map is of the shape $L(\U x) = M\U x^{\sigma T}$
for an $\F_q$-matrix $M_{n \times n}$. Clearly, $L$ is bijective if and only if $M$ is 
full-rank.
\item A semiquadratic map $S = (S_1,\ldots,S_n)$ is of the shape
$S_i(\U x) = \U xN^{(i)}\U x^{\sigma T}$
for $\F_q$-matrices $N^{(i)}_{n \times n}$ and $1 \le i \le n$. Its polarization
\[
	S(\U x+\U y) - S(\U x) - S(\U y) = B(\U x,\U y) + B(\U y,\U x)
\]
where $B = (B_1,\ldots,B_n)$ is the sesquilinear map defined by 
the sesquilinear forms $B_i(\U x,\U y) = \U xN^{(i)}\U y^{\sigma T}$.
\end{enumerate}
\end{remark}

In odd characteristic (i.e., when $q = p^m$ with $p$ an odd prime)  
there are no semiquadratic bijections since
$2 \mid \gcd(\sigma+1,q-1)$ in this case. 
Note that $\sigma \in \{p^0,p^1,\ldots,p^{m-1}\}$
is an integer and $x \mapsto x^\sigma$ is the actual 
exponentiation in the field $\F_{p^m}$.
Thus, the \textit{optimal} 
case corresponds to maps that are two-to-one on
$\F_q^n \setminus \{\U 0\}$ which are interesting combinatorially
as we will explain in the next section.

\subsection{$\GL(n,q)$-equivalence}

Let $F : \F_q^n \to \F_q^n$. The natural group actions 
$L \circ F$ and $F \circ L$ together define a notion of 
equivalence where $L \in \GL(n,q)$. We say that 
$F,G : \F_q^n \to \F_q^n$
are $\GL(n,q)$-equivalent if $G = L_1 \circ F \circ L_2$ for some
$L_1,L_2 \in \GL(n,q)$ and write $F \approx G$. 

\begin{remark}[$\F_{q^n}$-polynomials and change of basis]
Any $F : \F_q^n \to \F_q^n$ can be written uniquely as 
(the evaluation function of) a polynomial
\[
	\widetilde{F}(X) = \sum_{i = 0}^{q^n-1} A_i X^i \in \F_{q^n}[X]/(X^{q^n}-X),
\]
given a basis choice $\beta  = (\beta_1 ,\ldots,\beta_n )$  for $X    = (\beta_1     x_1+\cdots+\beta_n  x_n)$ 
and                  $\gamma = (\gamma_1,\ldots,\gamma_n)$  for $\widetilde{F}(X) = (\gamma_1 F(\U x)_1+\cdots+\gamma_n F(\U x)_n)$.
The notion of $\GL(n,q)$-equivalence is therefore that of basis change.
Thus, we may extend the use of the $\GL(n,q)$-equivalence between 
univariate and multivariate representations and write 
$\widetilde{F}(X) \approx F(x_1,\ldots,x_n)$.
We reserve the capital variable name $X$ for the 
univariate notation and the lower-case variable names
$x,y,z,\ldots$ for the multivariate case.

The following polynomials are the 
$\F_{q^n}[X]/(X^{q^n}-X)$-representations of semilinear 
and semiquadratic maps. 
\begin{enumerate}
\item A semilinear map is a polynomial of the shape
\[
	\widetilde{L}(X) = \sum_{1 \le i \le n} A_i X^{q^i\sigma},
\]
which are known as $\F_q$-linearized polynomials when $\sigma = \id$. 
\item A semiquadratic map is a polynomial of the shape
\[
    \widetilde{S}(X) = \sum_{1 \le i,j \le n} B_{ij} X^{q^i\sigma+q^j},
\]
which are known as Dembowski-Ostrom polynomials when $\sigma = \id$. 
(If $p=2$ and $\sigma=\id$, some authors require $i\neq j$.)

\end{enumerate}
\end{remark}

\subsection{Subfields of $\F_q$} By again choosing arbitrary bases, 
one can embed $\GL(n,q)$    and $\GammaL(n,q)$ in 
$\GL(nk,p^{\ell})$ and $\GammaL(nk,p^{\ell})$ respectively 
if $q = p^{m}$ and $m = k\ell$. If $\Fix_{q}(\sigma) = \F_{p^{\ell}}$ 
is the fixed field of $\sigma \in \Aut(\F_q)$ in $\F_q$ then 
a $\sigma$-semilinear map of $\F_{p^m}^n$ becomes
a linear map of $\F_{p^{\ell}}^{nk}$.
We have the chain of inclusions
\[
\GL(n,q) \le \GammaL(n,q) \le \GL(nm,p).
\]
Similar observations hold for semiquadratic maps:
\[
\QH(n,q) \subseteq \SQH(n,q) \subseteq \QH(nm,p).
\]

\subsection{Semifields} An $n$-dimensional non-associative 
(meaning \textit{not necessarily associative})
$\F_q$-algebra is the vector space $\F_q^n$ equipped with a bilinear 
multiplication $\U x \ast \U y$.  
By an observation of Dembowski and Ostrom \cite[p. 257]{DO}, 
the multiplication operation of a commutative $\F_q$-algebra 
(for odd $q$) can be written as the polarization
\[
	\U x \ast \U y = Q(\U x + \U y) -Q(\U x)-Q(\U y) 
\]
for a quadratic homogeneous map $Q \in \QH(n,q)$. We say that a 
finite algebra is a division algebra when the condition
$\U x \ast \U y = 0$ if and only if $\U x = \U 0$ or $\U y = \U 0$
holds. 
Instead of a ``finite not-necessarily associative unital division algebra'' 
we use the terminology \textit{finite semifield} introduced by 
Knuth~\cite{Knuth65}, which has become standard. 
Semifields have attracted a lot of interest in finite geometry since they 
can be used to coordinatize certain non-Desarguesian projective planes, 
see Section~\ref{s:planes} for details.
Menichetti proved (the Kaplansky conjecture) \cite{menichetti1977kaplansky} 
that a finite semifield three dimensional over its center
is isotopic to either a field or a twisted field of Albert 
(see Section \ref{s:isotopy} for details). 
In the commutative case this can be re-stated as follows. Note that a finite division algebra whose multiplication is not required to have an identity is called a pre-semifield.
If the polarization of a quadratic homogeneous map $Q \in \QH(3,q)$
gives a pre-semifield multiplication then 
$Q \approx X^2$ or $Q \approx X^{q+1}$ where 
Dembowski-Ostrom polynomial representations for $Q$ is used
(see Theorem \ref{thm_menichetti}).
Note that an $n$-dimensional $\F_{p^m}$-algebra
is an $nk$-dimensional $\F_{p^{\ell}}$-algebra if $m = k\ell$. 
The polarization of a $\sigma$-semiquadratic map 
$S \in \SQH(n,p^m)$ defines the multiplication of an 
$nk$-dimensional $\F_{p^{\ell}}$-algebra if 
$\Fix_{q}(\sigma) = \F_{p^{\ell}}$. Our aim in this paper is to 
identify a large class of finite pre-semifields whose 
multiplications are polarizations of 
$S \in \SQH(3,p^m)$ that arise from a non-trivial use of 
the skew-polynomials introduced by Ore \cite{Ore1}.

The main result of our paper is the following theorem.

\begin{theorem} \label{thm_main}
Let $q = p^m$ be odd and $F_1,F_2,F_3 \colon \F_q^3\rightarrow \F_q$ be defined as
\begin{align*}
F_1(x,y,z) &=  x^{\sigma+1} + ay^{\sigma}z + bx^{\sigma}y + cx^{\sigma}z,\\
F_2(x,y,z) &= ay^{\sigma+1} +  z^{\sigma}x + bz^{\sigma}y + cx^{\sigma}y,\\
F_3(x,y,z) &=  z^{\sigma+1} -  x^{\sigma}y,
\end{align*}
where $a,b,c \in \F_q$ and 
$\sigma : x \mapsto x^{p^k}$ with $m/\gcd(k,m)$ odd.

The polarization of $F = (F_1,F_2,F_3) \in \SQH(3,q)$ describes a pre-semifield
multiplication if and only if the equation	
\begin{equation} \label{eq:condition}
    x^{\sigma^2+\sigma+1} + cx^{\sigma^2+\sigma} + bx^{\sigma^2} + a = 0,
\end{equation}
has no solution $x \in \F_q$. 
\end{theorem}

It is well known that for any $\sigma$ there exist $a,b,c \in \F_q$ such 
that Eq.~\eqref{eq:condition} has no solution in $\F_q$. For details, we 
refer to Appendix~\ref{sec:lemmas}.

\subsection{Bijections of $\PN{n-1}(\F_q)$}
Define an equivalence relation among vectors 
$\U u,\U v \in \F_q^n \setminus \{\U 0\}$
as $\U u \sim \U v$ if and only if $\U u = \lambda \U v$ 
for some $\lambda \in \F_q^\times$. 
Recall that the projective space $\PN{n-1}(\F_q)$ is defined as the 
equivalence classes of $\F_q^n \setminus \{\U 0\}$ modulo the
equivalence relation $\sim$. Any map $F : \F_q^n \to \F_q^n$,
that satisfies $F(\lambda\U x) = \phi(\lambda)F(\U x)$ for 
$\lambda \in \F_q$ and $\phi : \F_q \to \F_q$, 
and $F(\U x) \ne 0$ for $\U x \in \F_q^n \setminus \{\U 0\}$ induces a 
well-defined map 
$\overline{F}$ on the projective space $\PN{n-1}(\F_q)$. 
The maps of $\GL$, $\GammaL$, $\QH$ and $\SQH$ are
$1$-,$\sigma$-, $2$-, and $(\sigma+1)$-homogeneous respectively, 
that is to say, $F(\lambda \U x) = \lambda^s F(\U x)$ is satisfied
for $s \in \{1,\sigma,2,\sigma+1\}$.

We define projective versions $\PQH(n,q)$ and $\PSQH(n,q)$
of $\QH$ and $\SQH$ analogous to $\PGL(n,q)$ and $\PGammaL(n,q)$
in relation to $\GL$ and $\GammaL$.

\begin{definition}
The projective homogeneous semiquadratic (resp. quadratic) maps
$\overline{F} \in \PSQH(n,q)$ (resp. $\PQH(n,q)$) 
are the maps
$\overline{F} : \PN{n-1}(\F_q) \to \PN{n-1}(\F_q)$
induced by $F \in \SQH(n,q)$ (resp. $\QH(n,q)$) 
that satisfy $F(\U x) = \U 0 \iff \U x = \U 0$.
\end{definition}
Even though the definition is quite natural, we are not aware of any 
previous work done on the set $\PSQH(n,q)$ in full generality. The 
special case of $\PSQH(2,q)$ plays (implicitly) a major role in recent 
classification results of certain permutation polynomials 
\cite{Gologlu22,DingZieve}. In this paper, we are interested in bijections 
in $\PSQH(3,q)$. The aforementioned classification of bijections in $\SQH(2,q)$ 
and $\PSQH(2,q)$ has only recently been done. We only give the $\PSQH$ case 
which implies the $\SQH$ result.

\begin{theorem}[Classification of bijections in $\PSQH(2,q)$ \cite{Gologlu22,DingZieve}]
Let $\sigma = p^k$, $q = p^n$, $d = \gcd(k,n)$ and
$\overline{F} \in \PSQH(2,q) \setminus \PQH(2,q)$ 
be a $\sigma$-semiquadratic map. 
Then $\overline{F}$ is bijective
if and only if
$p = 2$, and 
\begin{enumerate}
\item $n/d$ is even and $F \approx X \mapsto X^{\sigma+q}$, or 
\item $n/d$ is even and $F \approx X \mapsto X^{\sigma+1}$, or
\item $n/d$ is odd and $F \approx (x,y) \mapsto (x^{\sigma+1},y^{\sigma+1})$, or
\item $n/d$ is odd and 
\begin{enumerate}
\item $F \approx X \mapsto X^{\sigma+q}$ if $k/d$ is odd, or 
\item $F \approx X \mapsto X^{\sigma+1}$ if $k/d$ is even.
\end{enumerate}
\end{enumerate}

\end{theorem}

The case $\overline{F} \in \PQH(2,q)$ is immediate: 
$F(X) \approx X^2$ and $p = 2$ (see \cite[Lemma 1.2]{DingZieve22}).

The classification tells us that nothing surprising happens 
in dimension two: Bijections are either ``monomial'' 
in the Dembowski-Ostrom representation or they are
``degenerate'' (which we will detail in the next section). 
In this paper, we will give non-degenerate bijections
in $\PSQH(3,q)$ that are inequivalent to Dembowski-Ostrom monomials.
They are natural but non-trivial. We will prove the 
following theorem which is the projective version 
of Theorem \ref{thm_main}. In fact, Theorem \ref{thm_main}
follows from Theorem \ref{thm_main_v2}.
We repeat the statement
because the $\gcd$-condition is no longer needed 
and there is no restriction on the characteristic.

\begin{theorem} \label{thm_main_v2}
Let $q = p^m$ and $F_1,F_2,F_3 \colon \F_q^3\rightarrow \F_q$ be defined as
\begin{align*}
F_1(x,y,z) &=  x^{\sigma+1} + ay^{\sigma}z + bx^{\sigma}y + cx^{\sigma}z,\\
F_2(x,y,z) &= ay^{\sigma+1} +  z^{\sigma}x + bz^{\sigma}y + cx^{\sigma}y,\\
F_3(x,y,z) &=  z^{\sigma+1} -  x^{\sigma}y,
\end{align*}
where $a,b,c \in \F_q$ and 
$\sigma=p^k$.

Then $F$ induces a permutation $\overline{F} = (\overline {F_1,F_2,F_3})\in \PSQH(3,q)$  
if and only if the equation	
\[
	x^{\sigma^2+\sigma+1} + cx^{\sigma^2+\sigma} + bx^{\sigma^2} + a = 0
\]
has no solution $x \in \F_q$. 
\end{theorem}

The following theorem summarizes the connection between 
pre-semifields and $s$-to-one maps in $\SQH$ and $\PSQH$, and provides the link between Theorems~\ref{thm_main} and~\ref{thm_main_v2}.

\begin{theorem}\label{thm_sqh}
Let $q$ be odd, $F \in \SQH(n,q)$ 
with companion automorphism $\sigma \in \Aut(\F_q)$.
\begin{enumerate}
\item The polarization of $F$ describes a (commutative) 
pre-semifield multiplication if and only if 
\begin{enumerate}
\item $F(\U x) = \U 0 \iff \U x = \U 0$, and 
\item $F$ is $2$-to-$1$ on $\F_q^n \setminus \{\U 0\}$.
\end{enumerate}
\item  If (a) and (b) in (i) hold then
\begin{enumerate}
\item if $n$ is odd:  $\overline{F}$ is bijective on $\PN{n-1}(\F_q)$, and
\item if $n$ is even: $\overline{F}$ is $2$-to-$1$ on $\PN{n-1}(\F_q)$.
\end{enumerate}
\item If $\overline{F}$ is bijective on $\PN{n-1}(\F_q)$ then $F$ is 
$\gcd(\sigma+1,q-1)$-to-$1$ on $\F_q^n \setminus \{\U 0\}$.
\end{enumerate}
\end{theorem}

The first part connecting commutative pre-semifields and 
$2$-to-$1$ functions was proven by Weng and Zeng~\cite{weng2012further}, 
and Kyureghyan and Pott~\cite{KyuPott},
the second and third parts establishing the connection between $2$-to-$1$ functions 
in $\SQH(n,q)$ and bijections of the projective space will be proven in 
Section~\ref{sec:proof_bijections}. Together, they give (in odd characteristic 
and for odd $n$) the link between bijections of the projective space and commutative 
pre-semifields. In particular, Theorem~\ref{thm_sqh} immediately proves 
Theorem~\ref{thm_main} from Theorem~\ref{thm_main_v2}. 

The bijections of the projective plane $\P^2(\F_q)$ constructed in 
Theorem~\ref{thm_main_v2} thus yield a family of pre-semifields, which 
can then again be used to coordinatize a \emph{non-Desarguesian} 
projective plane (of the same order). Recall that the projective 
semilinear group $\PGammaL(3,q)$ is the collineation group of the 
projective plane $\P^{2}(\F_q)$, i.e., its elements are exactly the 
bijections of points and lines that preserve the point-line incidence 
of $\P^{2}(\F_q)$. What our results show is that the bijections of 
$\P^{2}(\F_q)$ in $\PSQH(3,q)$ create (via their polarization) new 
point-line incidence structures that are also projective planes which 
are \emph{non-isomorphic} to the Desarguesian plane. These connections, 
which are here developed for the first time, motivate the title of 
this paper. We are not aware of any previous work that studied the 
maps $\PSQH(n,q)$ in this context (or at all).

\subsection{Structure of the paper}
In Section \ref{sec_prelim} we give more background for 
Section \ref{sec_main} where we prove our main result 
 Theorem \ref{thm_main_v2}.
Then we proceed to show that the family of bijections/pre-semifields 
are indeed new by showing extended inequivalence/non-isotopy results
(Section \ref{sec_ineq}). In the appendix we compute the nuclei of
the semifields (which is an important invariant that measures
the level of non-associativity of an algebra) and also prove necessary
lemmas which are required in our theorems.

%%%%%%%%%%%%%%%%%%%%%%%%%%%%%%%%%%%%%%%%%%%%%%%%%%%%%%%%%%%%%%%%%%%%%%
%%%%%%%%%%%%%%%%%%%%%%%%%%%%%%%%%%%%%%%%%%%%%%%%%%%%%%%%%%%%%%%%%%%%%%

\section{Preliminaries}\label{sec_prelim}

\subsection{Finite semifields}
A finite \emph{semifield} $\S = (S,+,\circ)$ is a  set $S$ equipped with two operations $(+,\circ)$
satisfying the following axioms. 
\begin{enumerate}
\item[(S1)] $(S,+)$ is a finite group.
\item[(S2)] For all $x,y,z \in S$,
\begin{enumerate}
\item $x\circ (y+z) = x \circ y + x \circ z$,
\item $(x+y)\circ z = x \circ z + y \circ z$.
\end{enumerate}
\item[(S3)] For all $x,y \in S$, $x \circ y = 0$ implies $x=0$ or $y=0$.
\item[(S4)] There exists $\epsilon \in S$ such that $x\circ \epsilon = x = \epsilon \circ x$.
\end{enumerate}

Thus, a finite semifield is a finite non-associative unital division algebra.
This paper is only concerned with finite semifields 
and we omit the word ``finite'' from now on.
An algebraic object satisfying the first three of the above axioms is called 
a pre-semifield. If $\P = (P,+,\circ)$ is a pre-semifield, then $(P,+)$ is 
an elementary abelian $p$-group \cite[p. 185]{Knuth65}, and $(P,+)$ can be viewed as
an $n$-dimensional $\f{p}$-vector space $\F_p^n$. Defining a (pre-)semifield thus boils 
down to defining a bilinear (pre-)semifield multiplication on  $\F_p^n$.
If $\circ$ is associative then $\S$ is the finite field $\F_{p^n}$ by 
Wedderburn's (little) theorem.

A pre-semifield $\P = (\F_p^n,+,\circ)$ 
can be converted to a semifield $\S = (\F_p^n,+,\ast)$, for example using {\em Kaplansky's trick} 
by defining the new multiplication as
\[
	(x \circ e) \ast (e \circ y) = (x \circ y),
\]
for any non-zero element $e \in \F_p^n$, making $(e \circ e)$ the multiplicative 
identity of $\S$.

Two pre-semifields $\P_1 = (\F_p^n,+,\circ_1)$ and $\P_2 = (\F_p^n,+,\circ_2)$ 
are said to be \emph{isotopic} if there exist bijections $L,M,N \in \GL(n,p)$ 
satisfying
\[
N(x \circ_1 y) = L(x) \circ_2 M(y).
\]
Such a triple $\gamma = (N,L,M)$ is called an \emph{isotopism} 
from $\P_1$ to $\P_2$. Clearly, a pre-semifield $\P$ and its 
corresponding semifield $\S$ constructed by Kaplansky's trick 
are isotopic via $(\id,R_e,L_e)$ where $R_e,L_e : \F_p^n \to \F_p^n$ 
are right and left multiplication maps of $\P$. If $\P$ is 
commutative then $\S$ is commutative as well. 
Isotopisms between a pre-semifield $\P$ and itself 
are called \emph{autotopisms} and the set of autotopisms forms 
the \emph{autotopism group} of the pre-semifield $\P$, denoted by 
$\Autt(\P)$. 

\subsection{Projective planes} \label{s:planes}
We define the (Desarguesian) projective plane $\PN{2}(\F_q)$ via the equivalence 
relation on $\F_q^3 \setminus \{(0,0,0)\}$ defined by 
$(x,y,z) \sim (\lambda x,\lambda y,\lambda z)$ for $\lambda \in \F_q^\times$. 
We can thus identify $\PN{2}(\F_q)$ with the following set of homogeneous 
coordinates: 
\[
	\PN{2}(\F_q)=
	\left\{ (x,y,1) \colon x,y \in \F_q \right\} 
	\cup 
	\left\{ (x,1,0) \colon x \in \F_q \right\} 
	\cup 
	\left\{ (1,0,0) \right\}.
\]

As usual, 

\begin{enumerate}
\item the points $(x,y,1)$ will be called the \emph{affine points} of $\PN{2}(\F_q)$, and
\item $(1,0,0)$ the \emph{point at infinity}. 
\end{enumerate}

The $q^2+q+1$ lines of $\PN{2}(\F_q)$ are 
\begin{enumerate}
\item $l_\infty=\left\{(x,1,0) \colon x \in \F_q\right\} \cup \left\{(1,0,0)\right\}$ (the \emph{line at infinity}),
\item $l_{a,b}=\{(ax+b,x,1) \colon x \in \F_q\} \cup \{(a,1,0)\}$, and 
\item $l_a=\{(y,a,1) \colon y \in \F_q\} \cup \{(1,0,0)\}$ for $a,b \in \F_q$.
\end{enumerate}

Every pre-semifield can be used to construct a projective plane. This is done 
almost identically to the construction of the Desarguesian plane, with the only 
difference that the lines $l_{a,b}$ use the pre-semifield multiplication instead 
of the finite field multiplication, i.e., the lines are 
\[
l_{a,b}=\{(a \circ x +b ,x,1) \colon x \in \S\} \cup \{(a,1,0)\}. 
\]

Two pre-semifields define isomorphic projective planes if and only if they are 
isotopic~\cite{Albert60}. The pre-semifields constructed in Theorem~\ref{thm_main} 
from bijections of the Desarguesian plane $\P^2(\F_q)$ thus yield new 
(non-Desarguesian) projective planes using our inequivalence/non-isotopy
results of Section \ref{sec_ineq}.

\subsection{Commutative semifields} \label{subs:comm_semifields}
Not too many families of semifields are known, this is in particular the case 
for commutative semifields. The following is the list of all known 
infinite families of commutative semifields that works for  
orders $q^3$ where $q$ is an odd prime power for arbitrary characteristic.
\begin{enumerate}
\item Finite fields (every $q$).
\item The twisted fields of Albert (every $q$ except $q = p^{2^k}$ for some $k \in \N$). 
\item The Bierbrauer (BB) and Zha-Kyureghyan-Wang (ZKW) semifields
(more restrictive on $q$, see Section \ref{sec_ineq}).
\end{enumerate}
For a full list of known infinite families of commutative 
semifields in odd characteristic, 
see \cite[Tables 1 and 2]{golouglu2023exponential}.

\subsection{Polarization} 
The \emph{polarization} of a map  
$F \colon \F_q^n \rightarrow \F_q^n$ with $F(0)=0$ is defined as
\[
\Delta_F(x,y) = F(x+y) -F(x) -F(y). 
\]
If $F \in \QH(n,q)$ then $\Delta_F : \F_q^n \times \F_q^n \to \F_q^n$ 
is symmetric and $\f{q}$-bilinear. Thus if $\Delta_F(x,a) = 0$ implies $x = 0$ for all 
$a \in \F_{q^n}^\times$, then $\Delta_F(x,y)$ describes a commutative pre-semifield 
multiplication $x \circ y :=\Delta_F(x,y)$, cf. \cite{DO}.
If this happens, we call $F$ a \emph{planar} or \emph{perfectly nonlinear} function.
In characteristic two, this cannot happen since we have $\Delta_F(x,x) = 0$ for all $x$.
Conversely, every commutative pre-semifield multiplication in odd characteristic of order 
$p^N$ can be written as $\Delta_F(x,y)$ for some $F \in \QH(N,p)$.

\subsection{Bijections of the projective plane}

A homogeneous map $F$ of $\F_q^3$ 
that satisfies $F(x,y,z) = \U{0} \iff (x,y,z) = \U{0}$
induces a map $\overline{F}$ 
of $\PN{2}(\F_q)$ since $F \colon \F_q^3 \rightarrow \F_q^3$ 
defined via
\[
F(x,y,z)=(f(x,y,z),g(x,y,z),h(x,y,z)),
\]
satisfies $F(\lambda x,\lambda y , \lambda z)=\lambda^s F(x,y,z)$ 
for all $\lambda \in \F_q^\times$ and some positive integer $s$. 
We call such maps \emph{$s$-homogeneous}. Of course, this is equivalent to 
all of $f,g,h \colon \F_q^3 \rightarrow \F_q$ being $s$-homogeneous. 

To check whether $\overline F$ is bijective on $\PN{2}(\F_q)$ one needs to verify

\begin{gather}
	F(x,y,z) = (0,0,0) \iff (x,y,z) = (0,0,0), \label{eq_PP1} \\
	\left\{\frac{f(x,y,z)}{g(x,y,z)} \colon h(x,y,z) = 0 \right\} = \F_q \cup \{\infty\},\label{eq_PP2} \\
	\left\{\left(\frac{f(x,y,z)}{h(x,y,z)},\frac{g(x,y,z)}{h(x,y,z)}\right) 
	\colon h(x,y,z) \ne 0 \right\}=\F_q \times \F_q. \label{eq_PP3}
\end{gather}

\vspace{1em}

\subsection{Monomial bijections in $\PSQH$}
In the 1930s, Albert introduced the family of twisted fields.
Commutative twisted fields arise from Dembowski-Ostrom
monomials in $\SQH(1,q^n)$ that is to say
$A_\sigma(X) = X^{\sigma+1}$ where $q=p^m$ is odd, $\sigma = p^k$
with the condition that $nm/\gcd(k,nm)$ is odd.
We need the following well-known lemma.

\begin{lemma} \label{lem_gcd}
	Let $k,m \in \N$ and $p$ be a prime. Then
	\begin{itemize}
	\item $\gcd(p^k-1,p^m-1) = p^{\gcd(k,m)}-1$.
	\item $\gcd(p^k+1,p^m-1)=\begin{cases}
		1 & \text{if } m/\gcd(k,m) \text{ odd, and } p=2, \\
		2 & \text{if } m/\gcd(k,m) \text{ odd, and } p>2, \\
		p^{\gcd(k,m)}+1 & \text{if } m/\gcd(k,m) \text{ even}.
	\end{cases}$
\end{itemize}
\end{lemma}

Now, $A_\sigma$ is two-to-one on $\F_{q^n}^\times$ if and only if
$p$ and $nm/\gcd(k,nm)$ are odd. Trivially, $\overline{A_\sigma} \in \PSQH(1,q^n).$
We can choose $\F_{q}$-bases of $\F_{q^n}$ and embed $A_\sigma$ in $\SQH(n,q)$. 
Now it is clear that the Albert monomials 
$\overline{A_\sigma} \in \PSQH(n,q)$ 
when $nm/\gcd(k,nm)$ is odd,
are
\begin{enumerate} 
\item bijective on $\PN{n-1}(\F_{q})$ if $n$ is odd, and
\item $2$-to-$1$ on $\PN{n-1}(\F_{q})$ if $n$ is even,
\end{enumerate}
by inspecting $\gcd(\sigma+1,(q^n-1)/(q-1))$.
Theorem~\ref{thm_sqh} (ii) shows that this is also the case for general $\SQH(n,q)$. 

Note that $A_{\id}(X) = X^2 \approx X^2/2 \in \QH(1,q^n)$ and 
$\Delta_{A_{\id}}$ is \textit{isotopic} to $X Y$, i.e., the finite field.
The corresponding projective plane arising from a skew-field is called
the Desarguesian plane. We proceed with the definition of isotopy.

\subsection{Isotopy of semifields and equivalence of functions} \label{s:isotopy}
As discussed in the introduction, all functions $F \in \SQH(n,q)$ can be 
naturally embedded in $\QH(nm,p)$ by fixing an $\F_p$-basis for $\F_q$, 
where $q=p^m$. The coarsest equivalence notion on $\SQH(n,q)$ we can 
consider is thus $\GL(nm,p)$-equivalence. For brevity, we will refer to this 
equivalence relation on functions as \emph{p-equivalence}.
\begin{definition}
Two functions $F_1,F_2 : \F_p^{nm} \to \F_p^{nm}$ are said to be $p$-equivalent if
\[
    F_1 \approx_p F_2 \iff F_2 = L_1 \circ F_1 \circ L_2
\]
for some $L_1,L_2 \in \GL(nm,p)$.
\end{definition}

We are interested in the interplay between $\GL(n,q)$-equivalence and $p$-equivalence 
of two planar functions $F,G \in \SQH(n,q)$ and isotopy of the pre-semifields 
$(\S_F,+,\ast_F)$, $(\S_G,+,\ast_G)$ induced by $F$ and $G$. It was proved in 
\cite[Theorem 3.5]{coulterhenderson} that $F$ and $G$ are $p$-equivalent via 
$L_1 \circ F=G\circ L_2$ if and only if $L_1(x \ast_F y)=L_2(x)\ast_G L_2(y)$. 
Thus, $p$-equivalence of planar functions implies isotopy of the corresponding 
semifields. In fact, this type of isotopy is called \textit{strong isotopy}. 
The converse direction is more delicate. It is possible that two planar functions 
associated to isotopic semifields are not $p$-equivalent 
(cf. \cite{ZP13} for an example). However, every isotopy 
class can produce at most two planar functions that are pairwise $p$-inequivalent. 
Exact conditions when a one-to-one correspondence between isotopy of semifields and 
$p$-equivalence of planar functions occurs are described in
\cite[Theorem 2.6]{coulterhenderson}.

Menichetti proved \cite{menichetti1977kaplansky} that any semifield at most three dimensional
over its \textit{center} $\F_q$ (see the precise definition in Appendix~\ref{sec_nuclei}),
that is to say, whose multiplication can be written as
\[
	X \ast Y = \sum_{0 \le i,j \le 2} a_{i,j} X^{q^i}Y^{q^j}, a_{i,j} \in \F_{q^3},
\]
is isotopic to either $XY$ or $XY+bX^{q^i}Y^{q^j}$ with $b \in \F_{q^3}$. 
Kantor \cite[Proposition 5.3]{Kantor03} classified all commutative semifields of this form. 
Using the connection between the notions of equivalence outlined above we get the following result.
\begin{theorem}[The Menichetti-Kaplansky Theorem for quadratic maps] 
\label{thm_menichetti}
Let $F \in \QH(3,q)$ be planar. Then $F \approx X^2$ or $F \approx X^{q+1}$.
\end{theorem}

A main achievement of this paper (see Theorem~\ref{thm:inequiv}) is to show that the situation is completely different when considering $F \in \SQH(3,q)$, as Theorems~\ref{thm_main} and~\ref{thm_main_v2} yield many new semifields that are not equivalent to semifields induced by Dembowski-Ostrom monomials.

\subsection{Degenerate bijections}

Let $q = p^m$ and $\sigma = p^k$. In even characteristic, 
$\gcd(\sigma+1,q-1) = 1$ can happen when $m/\gcd(k,m)$ is odd 
(see Lemma \ref{lem_gcd}). In this case there are many ``degenerate'' 
semiquadratic maps in $\SQH(n,q)$ that are bijections, 
for instance
\begin{enumerate}
\item $(x,y,z) \mapsto (x^{\sigma+1},y^{\sigma+1},z^{\sigma+1})$, or
\item $(x,y,z,t) \mapsto (T_\beta^{-1}(T_\beta(x,y,z)^{\sigma+1})),t^{\sigma+1})$
\end{enumerate}
where in (ii) we identify $\F_q^3$ with the extension field $\F_{q^3}$ 
using $T_\beta :  (x,y,z) \mapsto \beta_1 x + \beta_2 y + \beta_3 z $
and apply a bijection in $\SQH(1,q^3)$ and then embed back to $\F_q^3$.

This suggests the following definition.

\begin{definition} We say that a semiquadratic map 
$S(\U{x}) = (S_1(\U{x}),\ldots,S_n(\U{x}))\in\SQH(n,q)$ 
is degenerate if a $\GL(n,q)$-equivalent 
semiquadratic map 
$T(\U{x}) = (T_1(\U{x}),\ldots,\allowbreak T_n(\U{x}))$
contains a semiquadratic form $T_j$ which is independent of a variable $x_i$.
We say that $\overline{S} \in \PSQH(n,q)$ is degenerate if $S$ is degenerate.
\end{definition}

We are going to show that a projective semiquadratic bijection when $p$ is odd
cannot be degenerate. This is surely not the case for $p = 2$ as we
mentioned above.

\begin{proposition}
Let $q=p^m$ and $S \in \SQH(n,q)$ with companion automorphism $\sigma$. 
If $\overline{S} \in \PSQH(n,q)$ is bijective on $\PN{n-1}(\F_q)$ and degenerate 
then $\gcd(\sigma+1,q-1)=1$, in particular $p=2$.
\end{proposition}
\begin{proof}
If $\overline{S}$ is degenerate, than so is $S \in \SQH(n,q)$ by definition.
Let, if necessary after (left and right) composition with some $L_1,L_2 \in \GL(n,q)$, 
$S = (S_1,\ldots,S_n) \in \SQH(n,q)$ have (w.l.o.g.) $S_n$ independent of the variable
(w.l.o.g.) $x_n$. Then 
\[
	S_n(\U x+\U u) - S_n(\U x) - S_n(\U u) = 0
\]
when $\U {u} = (0,\ldots,0,b)$ for some $b \in \F_q^\times$. Thus 
the $\F_p$-linear function
\[
	L_{\U u}=S(\U{x}+\U{u}) - S(\U{x}) - S(\U{u}) 
\]
has $\F_p$-rank at most $m(n-1)$.
Consider now the $\F_p$-affine map 
\[
	D_{\U u} : \U x \mapsto S(\U x+\U u) - S(\U x).
\]
By the previous considerations, every element in the image set of 
$D_{\U u} $ has at least $q$ preimages. Assume now that 
$\gcd(\sigma +1,q-1)>1$. Then we have a $1\neq c \in \F_q$ 
such that $c^{\sigma+1}=1$. Then 
\[
D_{\U u} \left(\frac{1}{c-1}\U u\right)
          =S\left(\frac{c}{c-1}\U u\right)-S\left(\frac{1}{c-1}\U u\right)
          =c^{\sigma+1}S\left(\frac{1}{c-1}\U u\right)-S\left(\frac{1}{c-1}\U u\right) = 0.
\]
So, $0$ is in the image set of $D_u$, which means $0$ has at 
least $q$ preimages. Since $q>q-1$, this means $\overline{S}$ cannot 
be bijective. We conclude that $\gcd(\sigma +1,q-1)=1$. 
\end{proof}

%%%%%%%%%%%%%%%%%%%%%%%%%%%%%%%%%%%%%%%%%%%%%%%%%%%%%%%%%%%%%%%%%%%%%%
\section{Proof of the main result}\label{sec_main}

We start this section by proving Theorem~\ref{thm_sqh}. 

\subsection{Proof of Theorem~\ref{thm_sqh}}\label{sec:proof_bijections}
Let us repeat the statement of Theorem~\ref{thm_sqh} here.
We will need a generalization of 
\cite[Theorem 1.8]{weng2012further} of Weng and Zeng 
which we provide in the appendix (see Proposition \ref{prop_wengzeng}).

\begin{theorem*}
Let $q$ be odd, $F \in \SQH(n,q)$
with companion automorphism $\sigma \in \Aut(\F_q)$.
\begin{enumerate}
\item The polarization of $F$ describes a (commutative)
pre-semifield multiplication if and only if
\begin{enumerate}
\item $F(\U x) = \U 0 \iff \U x = \U 0$, and
\item $F$ is $2$-to-$1$ on $\F_q^n \setminus \{\U 0\}$.
\end{enumerate}
\item  If (a) and (b) in (i) hold then
\begin{enumerate}
\item if $n$ is odd:  $\overline{F}$ is bijective on $\PN{n-1}(\F_q)$, and
\item if $n$ is even: $\overline{F}$ is $2$-to-$1$ on $\PN{n-1}(\F_q)$.
\end{enumerate}
\item If $\overline{F}$ is bijective on $\PN{n-1}(\F_q)$ then $F$ is
$\gcd(\sigma+1,q-1)$-to-$1$ on $\F_q^n \setminus \{\U 0\}$.
\end{enumerate}
\end{theorem*}

\begin{proof}
\begin{enumerate}
\item[]
\item $\Leftarrow$ proved by Weng and Zeng ~\cite[Theorem 2.3]{weng2012further} and 
$\Rightarrow$ proved by Kyureghyan and Pott ~\cite[Corollary 1]{KyuPott}.

\item Assume that $n$ is odd, and assume further that 
$\overline{F}(\U{x_1})=\overline{F}(\U{x_2})$
for $\U{x_1},\U{x_2} \in \F_q^n \setminus \{\U 0\}$
i.e., 
$F(\U{x_2})=\lambda F(\U{x_1})$ for some 
$\lambda \in \F_q^\times$. 
We will show that $\U{x_1} \sim \U{x_2}$ 
i.e., they correspond the the same point in $\PN{n-1}(\F_q)$. 
By Proposition~\ref{prop_wengzeng}, $\lambda$ is necessarily 
a square, say $\lambda=\mu^{\sigma+1}$ for some $\mu \in \F_q^\times$ 
since the set of squares in $\F_q^\times$ is exactly the set of 
$(\sigma+1)$-st powers since $F$ is $2$-to-$1$ on $\F_q^\times$. 
So $F(\U{x_2})=\mu^{\sigma+1} F(\U{x_1})=F(\mu \U{x_1})$. 
Since $F(\U{x})=F(-\U{x})$ for all $x \in \F_q^n$ 
(for $F \in \SQH(n,q)$) 
and $F$ is $2$-to-$1$ by assumption, 
we necessarily get $\U{x_2}= \pm \mu \U{x_1}$. 
		
Now consider the case that $n$ is even. Let $\mu$ be a non-square 
in $\F_q^\times$ and $\U{x_1}\in \F_q^n\setminus\{\U{0}\}$.
By Proposition~\ref{prop_wengzeng}, there exists 
an (up to $\pm$) unique $\U{x_2} \in \F_q^n \setminus \{\U 0\}$, such that 
$F(\U{x_2})=\mu F(\U{x_1})$, 
in particular 
$\overline{F}(\U{x_1})=\overline{F}(\U{x_2})$. 
Now suppose for the sake of contradiction that 
$\U{x_1} \sim \U{x_2}$. 
Then $\U{x_2}=\lambda\U{x_1}$ for some 
$\lambda \in \F_q^\times$, so 
$F(\U{x_2})=\lambda^{\sigma+1}F(\U{x_1})$, 
and $\lambda^{\sigma+1}=\mu$ which is a non-square in $\F_q^\times$, 
leading to the desired contradiction. 
     
\item  Suppose $F(\U{x_1})=F(\U{x_2})$ 
for $\U{x_1},\U{x_2} \in \F_q^n \setminus \{\U 0\}$. 
Then clearly 
$\overline{F}(\U{x_1})=\overline{F}(\U{x_2})$,
and since $\overline{F}$ is bijective on $\PN{n-1}(\F_q)$, 
we have $\U{x_2}=\lambda \U{x_1}$ for some 
$\lambda \in \F_q^\times$. So 
$F(\U{x_2})=\lambda^{\sigma+1}F(\U{x_1})=F(\U{x_1})$ 
if and only if $\lambda^{\sigma+1}=1$, 
i.e., $\lambda$ is a $(\sigma+1)$-st root of unity in $\F_q^\times$, 
and the number of such roots of unity is exactly $\gcd(\sigma+1,q-1)$. 
\end{enumerate}
\end{proof}

\subsection{Proof of Theorem~\ref{thm_main_v2}}
Now we prove Theorem~\ref{thm_main_v2}. Let us repeat the statement.

\begin{theorem*} 
Let $q = p^m$ and $f,g,h \colon \F_q^3\rightarrow \F_q$ be defined as
\begin{align*}
f(x,y,z) &=  x^{\sigma+1} + ay^{\sigma}z + bx^{\sigma}y + cx^{\sigma}z,\\
g(x,y,z) &= ay^{\sigma+1} +  z^{\sigma}x + bz^{\sigma}y + cx^{\sigma}y,\\
h(x,y,z) &=  z^{\sigma+1} -  x^{\sigma}y,
\end{align*}
where $a,b,c \in \F_q$ and 
$\sigma=p^k$.

Then $F$ induces a permutation $\overline{F} = (\overline {f,g,h})\in \PSQH(3,q)$ 
if and only if the equation	
\[
	x^{\sigma^2+\sigma+1} + cx^{\sigma^2+\sigma} + bx^{\sigma^2} + a = 0
\]
has no solution $x \in \F_q$. 
\end{theorem*}

\begin{proof}
We need to verify Equations \eqref{eq_PP1}--\eqref{eq_PP3}.
Obviously, $F(0,0,0) = (0,0,0)$.
First assume $y \in \F_q$ is a solution to $x^{\sigma^2+\sigma+1} + cx^{\sigma^2+\sigma} + bx^{\sigma^2} + a = 0$. If $y=0$ then necessarily $a=0$ and $F(0,1,0)=(0,0,0)$, so $\overline{F}$ does not permute $\PN{2}(\F_q)$. If $y\neq 0$ then $F(y,1/y^\sigma,1)=(0,0,0)$ giving the same contradiction. So let us for the rest of the proof assume that $x^{\sigma^2+\sigma+1} + cx^{\sigma^2+\sigma} + bx^{\sigma^2} + a = 0$ has no solution $x \in \F_q$.

We first handle the case where two of $x,y,z$ vanish.
\begin{enumerate}
\item $y = z = 0$ implies $F(1,0,0) = (1,0,0)$.
\item $x = z = 0$ implies $F(0,1,0) = (0,a,0)$.
\item $x = y = 0$ implies $F(0,0,1) = (0,0,1)$,
\end{enumerate}
where (in $\PN{2}(\F_q)$) we have $(0,a,0) \sim (0,1,0)$. \\

\textbf{Case $h(x,y,z) = 0$:}

(i) If $z = 0$ then either $x = 0$ or $y = 0$ which were handled just before.

(ii) Let $z=1$. We consider the coordinates of $F(x,y,1)$:
\begin{align*}
f(x,y,1) &= x^{\sigma+1}  + ay^{\sigma} + bx^{\sigma}y + cx^{\sigma},\\
g(x,y,1) &= ay^{\sigma+1} +  x   + by    + cx^{\sigma}y,\\
h(x,y,1) &= 1 - x^{\sigma}y.
\end{align*}
We have $1 - x^{\sigma}y = 0$. Thus $xy \ne 0$ and $y = 1/x^{\sigma}$. Now
we want
\[
\left\{ \frac{f(x,y,1)}{g(x,y,1)} \ : \ h(x,y,1) = 0 \right\} = \F_q^\times,
\]
since $f(0,y,0)/g(0,y,0) = 0$ and 
      $f(x,0,0)/g(x,0,0) = \infty$.
Now
\[
\frac{f(x,1/x^{\sigma},1)}{g(x,1/x^{\sigma},1)} = \frac{ x^{\sigma^2+\sigma+1} + cx^{\sigma^2+\sigma} + bx^{\sigma^2} + a}{ x^{\sigma^2}  }
									\frac{ x^{\sigma^2+\sigma}}{ x^{\sigma^2+\sigma+1} + cx^{\sigma^2+\sigma} + bx^{\sigma^2} + a}
\]
permutes $\F_q^\times$ for $x \in \F_q^\times$ since $ x^{\sigma^2+\sigma+1} + cx^{\sigma^2+\sigma} + bx^{\sigma^2} + a \ne 0$ for all $x \in \F_q$. We conclude that 
\[
\{ f(x,y,z)/g(x,y,z) \ : \ h(x,y,z) = 0 \} = \F_q \cup \{\infty\},
\]
proving Eq.~\eqref{eq_PP2}. We also see that $f(x,y,z)=g(x,y,z)=h(x,y,z)=0$ does not occur for $(x,y,z)\neq (0,0,0)$, proving Eq.~\eqref{eq_PP1}.

\textbf{Case $h(x,y,z) \neq 0$:}

We have $h(x,y,z)=z^{\sigma+1} -x^{\sigma} y \ne 0$. We will prove Eq.~\eqref{eq_PP3}, i.e.
\[
	\left\{\left(\frac{f(x,y,z)}{h(x,y,z)},\frac{g(x,y,z)}{h(x,y,z)}\right) 
	\colon  h(x,y,z) \neq 0 \right\}=\F_q \times \F_q.
\]
First recall that the case $x=y=0$, $z=1$ yields the element $(0,0)$ in the set above. Using the homogeneity of $F$, it suffices to consider the cases $z=0$ and $z=1$.

(i) If $z = 0$, then $xy \ne 0$. Now,
\begin{align*}
f(x,1,0) &= x^{\sigma+1} + bx^{\sigma},\\
g(x,1,0) &= a + cx^{\sigma},\\ 
h(x,1,0) &= -x^{\sigma}.
\end{align*}
Then the set
\begin{align*}
S_Z &= 
\left\{ \left(\frac{f(x,1,0)} {h(x,1,0)} ,\frac{g(x,1,0)}{h(x,1,0)}  \right) \ : \ 
h(x,1,0) \ne 0 \right\} \\&= 
\left\{ \left(-u-b,-c-\frac{a}{u^{\sigma}} \right) \ : \ u \in \F_q^\times \right\}=\left\{ \left(u-b,\frac{a}{u^{\sigma}}-c \right) \ : \ u \in \F_q^\times \right\}
\end{align*}
collects the $q-1$ distinct pairs belonging to this case.

(ii) Let $z=1$ and 
\begin{align*}
f(x,y,1) &= x^{\sigma+1}  + ay^{\sigma} + bx^{\sigma}y + cx^{\sigma},\\
g(x,y,1) &= ay^{\sigma+1} +  x   + by    + cx^{\sigma}y,\\
h(x,y,1) &= 1 - x^{\sigma}y.
\end{align*}

It remains to show that the fractions
\begin{align}
\frac{f(x,y,1)}{h(x,y,1)} &= u\label{eq:quotient1}\\
\frac{g(x,y,1)}{h(x,y,1)} &= v \label{eq:quotient2}
\end{align}
have a common solution (in $(x,y)$) for all $(u,v) \in (\F_q \times \F_q) \setminus (S_Z\cup \{(0,0)\})$.
%then $F$ permutes $\PN{2}(\F_q)$ by the pigeonhole principle.
%Indeed, note that there are $q^2-q$ such $(u,v)$ pairs and the number of
%$(x,y) \in \F_q \times \F_q$ such that $h(x,y,1) = 0$ is also 
%$q^2-q+1$ since $x^{\sigma}y = 1$ has solutions for $(x,\frac{1}{x^{\sigma}})$
%whenever $x \in \F_q^\times$. 

Let us start with the case $x=0$, $y\neq 0$  (recall the case $x=y=0$ was already covered in the beginning). For $x=0$, the equations
simplify to
\begin{align*}
\frac{f(0,y,1)}{h(0,y,1)} &=ay^{\sigma}=u \\
\frac{g(0,y,1)}{h(0,y,1)} &=ay^{\sigma+1}+by = v.
\end{align*}
Let $S_X=\{(au^{\sigma},au^{\sigma+1}+bu) \colon u \in \F_q^\times\}$ collect these $q-1$ distinct solutions. 

%Similarly for $y=0$:
%\begin{align*}
%x^{\sigma+1}+cx^{\sigma}=&u \\
%x &= v.
%\end{align*}
%So $S_Y=\{(x^{\sigma+1}+cx^{\sigma},x) \colon x \in \F_q^\times\}$. We will thus assume $xy\neq 0$. 

We prove that $S_Z \cap S_X = \emptyset$. Assume that $(s_1,s_2) \in S_Z \cap S_X$. Then $s_1=ay^{\sigma}=x-b$ and $s_2=a/x^{\sigma}-c=ay^{\sigma+1}+by$. This yields $y^{\sigma}=\frac{x-b}{a}$. The equation in the second coordinate is equivalent to
\[a^{\sigma}-c^{\sigma}x^{\sigma^2} = a^{\sigma}y^{\sigma^2+\sigma}x^{\sigma^2}+b^{\sigma}y^{\sigma}x^{\sigma^2},\]
and after eliminating $y$ and multiplying by $a$ we get
\[x^{\sigma^2+\sigma+1}-bx^{\sigma^2+\sigma}+ac^{\sigma}x^{\sigma^2}-a^{\sigma+1}=0.\]
This equation has no solution by Lemma~\ref{lem:auxiliary}. We conclude that $S_Z \cap S_X=\emptyset$.

%if and only if $y^{\sigma^2+\sigma+1} = 1/a^{\sigma}$ which is impossible;
%\item that $S_Z \cap S_Y = \emptyset$ since for $y \ne 0$,
%\[
%\frac{a}{(y^{\sigma+1})^{\sigma}} = y
%\]
%if and only if $y^{\sigma^2+\sigma+1} = a$ which is impossible; and finally
%\item that $S_X \cap S_Y = \emptyset$ since for $y \ne 0$,
%$(ay^{\sigma+1})^{\sigma+1} = ay^{\sigma}$ if and only if $y^{\sigma^2+\sigma+1} = 1/a^{\sigma}$ 
%which is impossible.
%\end{itemize}

We now consider the case $x \neq 0$. Now, the two equations Eq.~\eqref{eq:quotient1} and Eq.~\eqref{eq:quotient2} become
\begin{align}
& x^{\sigma+1} + ay^{\sigma} + bx^{\sigma}y+cx^{\sigma}+ux^{\sigma}y    - u   = 0, \label{eq:beforetrans}\\
&ay^{\sigma+1} +  x   + by + cx^{\sigma}y+vx^{\sigma}y    - v   = 0.\label{eq:beforetrans2}
\end{align}

Set $t=yx^\sigma$ and eliminate $y$ from  Eqs.~\eqref{eq:beforetrans} and~\eqref{eq:beforetrans2}. Note that $t$ ranges in $\F_q\setminus \{1\}$ since $t=1$ is equivalent to $x^{\sigma}y=1$ and thus $h(x,y,z)=0$.  So for each pair
$(u,v) \in (\F_q \times \F_q) \setminus  (S_X \cup S_Z\cup \{(0,0)\})$, the equations
\begin{align}
& x^{\sigma^2+\sigma+1} + at^{\sigma}     + btx^{\sigma^2}+cx^{\sigma^2+\sigma}+x^{\sigma^2}   u (t-1) = 0, \label{eq_1}\\
& x^{\sigma^2+\sigma+1} + at^{\sigma+1} + btx^{\sigma^2}+ctx^{\sigma^2+\sigma}+x^{\sigma^2+\sigma} v (t-1) = 0. \label{eq_2}
\end{align}
should have a solution 
$(x,t) \in \F_q^\times \times (\F_q \setminus \{1\})$, which is then necessarily unique. From Eq.~\eqref{eq_1}, we immediately deduce that $t=1$ can never yield a solution since otherwise $x^{\sigma^2+\sigma+1}+cx^{\sigma^2+\sigma}+bx^{\sigma^2}+a=0$ has a solution, contradicting the assumption.
Now subtracting Eq.~\eqref{eq_1} from Eq.~\eqref{eq_2}, we get
\begin{equation} \label{eq:next1}
(1-t) (at^{\sigma} +cx^{\sigma^2+\sigma}+x^{\sigma^2+\sigma}v - x^{\sigma^2}u) = 0.
\end{equation}
Also, subtracting Eq.~\eqref{eq_2} from $t$ times Eq.~\eqref{eq_1},
we get
\begin{equation}\label{eq:next2}
(1-t) (x^{\sigma^2+\sigma+1} + btx^{\sigma^2}+tx^{\sigma^2}u - x^{\sigma^2+\sigma}v) = 0.
\end{equation}
Let us first check the case $u=-b$. 
Note that $(u,0)=(-b,0) \in S_X \cup \{(0,0)\}$, so $v\neq 0$ since $(u,v) \notin S_X \cup \{(0,0)\}$.
Then Eq.~\eqref{eq:next2} simplifies to $x^{\sigma^2+\sigma+1}-vx^{\sigma^2+\sigma}=0$, which has the unique solution $x=v \in \F_q^\times$. The unique solution for $y$ (and thus for $t$) can then be derived from Eq.~\eqref{eq:beforetrans}. So, we get unique solutions $(x,t) \in \F_q^\times \times (\F_q \setminus \{1\})$ for any $(-b,v)$.
 
So let us now consider the case $u \neq -b$. Then the two Eqs.~\eqref{eq:next1} and~\eqref{eq:next2} simplify to 
\begin{align*}
t^{\sigma} &= x^{\sigma^2} \frac{u -x^{\sigma}v-cx^{\sigma}}{a}, \\
t   &= x^{\sigma}     \frac{v-x}{u+b}    .
\end{align*}
Hence, we get
\[
	\frac{v^{\sigma}-x^{\sigma}}{u^{\sigma}+b^{\sigma}} = \frac{u -x^{\sigma}v-cx^{\sigma}}{a},
\]
which implies
\begin{equation}
	x^{\sigma} (u^{\sigma}v +b^{\sigma}v+u^{\sigma}c+b^{\sigma}c- a) = u^{\sigma+1} +b^{\sigma}u- v^{\sigma}a.
\label{eq:end}
\end{equation}

We investigate when $u^{\sigma+1} +b^{\sigma}u- v^{\sigma}a=0$. For each fixed $u$, there is a unique solution $v$, so there are exactly $q$ solutions. A simple calculation shows that these are precisely the points of $S_X \cup \{(0,0)\}$. Similarly, $u^{\sigma}v +b^{\sigma}v+u^{\sigma}c+b^{\sigma}c- a=0$ has in total $q-1$ solutions which are all covered by $(u,v) \in S_Z$.  
%The case $u^{\sigma}v +b^{\sigma}v+u^{\sigma}c+b^{\sigma}c- a=0$ corresponds to $(u,v) \in S_Z$ and
%$u^{\sigma+1} +b^{\sigma}u- v^{\sigma}a=0$ corresponds to    $(u,v) \in S_X$, 
So we can uniquely solve for $x \in \F_q^\times$ in Eq.~\eqref{eq:end} via
\[
	x^{\sigma} = \frac{u^{\sigma+1} +b^{\sigma}u- v^{\sigma}a}{u^{\sigma}v +b^{\sigma}v+u^{\sigma}c+b^{\sigma}c- a},
\]
so we find a solution $x \in \F_q^\times$ for all $(u,v) \in (\F_q \times \F_q) \setminus (S_X \cup S_Z \cup \{(0,0)\})$.
\end{proof}

%%%%%%%%%%%%%%%%%%%%%%%%%%%%%%%%%%%%%%%%%%%%%%%%%%%%%%%%%%%%%%%%%%%%%%%%%%%%%%%%%%%%%%

\section{Inequivalence to monomials and known semifields}\label{sec_ineq}

As our second result, we will show that the planar functions 
we present in Theorem~\ref{thm_main_v2} are inequivalent to 
Dembowski-Ostrom monomials and that the finite semifields 
associated to the pre-semifields
found in Theorem~\ref{thm_main} contain all known infinite 
families of semifields with $q^3$ elements in 
arbitrary characteristic; these are the Bierbrauer and the 
Zha-Kyureghyan-Wang semifields 
(see Section~\ref{subs:comm_semifields}). We see that both 
these families are contained as very special cases in our 
family. The goal of this section is to prove the following 
comprehensive inequivalence result.

\begin{theorem} \label{thm:inequiv}
Let $F \in \SQH(3,q)\setminus \QH(3,q)$ be a map 
from Theorem~\ref{thm_main_v2} with companion field 
automorphism $\sigma \ne \id$ and let $q=p^m$. Then,
\begin{enumerate}
\item $F$ is not $\GL(3,q)$-equivalent to any monomial.

\item If $F$ is planar and $m>2$, then $F$ is not $p$-equivalent 
to any monomial. Equivalently, the pre-semifield constructed by 
$F$ is not isotopic to the finite field or a twisted field.

\item The family of semifields associated to the pre-semifields 
constructed in Theorem~\ref{thm_main} contains 
\begin{enumerate}
\item all Zha-Kyureghyan-Wang semifields, and
\item all Bierbrauer semifields.
\end{enumerate}
		
\item The family of semifields associated to the pre-semifields 
constructed in Theorem~\ref{thm_main} contains	semifields 
that are not contained in any other known family of commutative semifields. 
\end{enumerate}
\end{theorem}
\begin{remark}
\begin{enumerate}
\item When we say that ``a family $\mathcal{F}_1$ of semifields contains 
a member of a family $\mathcal{F}_2$,'' we mean that $\mathcal{F}_1$ 
contains a member that is isotopic to the member in the family 
$\mathcal{F}_2$.

\item Recall that we defined the notion of $\GL(n,q)$-equivalence for 
two maps $Q_1,Q_2 \in \SQH(n,q)$ 
that maps $Q_1,Q_2 : \F_q^n \to \F_q^n$ where $q = p^m$ as
\[
	Q_1 \approx Q_2 \iff Q_2 = L_1 \circ Q_1 \circ L_2
\]
for some $L_1,L_2 \in \GL(n,q)$. 
A function $F=(f,g,h)$ from Theorem~\ref{thm_main_v2} is $\GL(3,q)$-equivalent to a function $G$ defined via a polynomial $G(X)\in \F_{q^3}[X]/(X^{q^3}-X)$ if and only if there exist two  $\F_q$-bases of $\F_q^3$, say $\{v_1,v_2,v_3\}$ and $\{u_1,u_2,u_3\}$ such that $G(xv_1+yv_2+zv_3)=f(x,y,z)u_1+g(x,y,z)u_2+h(x,y,z)u_3$ (as polynomials in $\F_{q^3}[x,y,z]/(x^q-x,y^q-y,z^q-z)$) since composition with a non-singular linear transformation  corresponds to a change of basis. 
Throughout the proof, let $f,g,h \colon \F_q^3 \rightarrow \F_q$ be the functions from Theorem~\ref{thm_main_v2}.
\end{enumerate}
\end{remark}
\begin{proof}[Proof of Theorem~\ref{thm:inequiv}] 
\begin{enumerate}
\item The only monomials (up to equivalence) we have to consider 
are $G(X) \in \{ X^{\sigma+1}, X^{\sigma q+1}, X^{\sigma q^2+1} \}$. 
We only prove here the first case. The proofs for others are essentially 
the same.

Assume that $G(xv_1+yv_2+zv_3)=f(x,y,z)u_1+g(x,y,z)u_2+h(x,y,z)u_3$  where 
$\{v_1,v_2,v_3\}$ and $\{u_1,u_2,u_3\}$ are $\F_q$-bases of $\F_q^3$. 
Calculating $G(xv_1+yv_2+zv_3)=(xv_1+yv_2+zv_3)^{\sigma+1}$, we get a 
non-vanishing monomial $xy^\sigma v_1v_2^\sigma$, but 
$f(x,y,z)u_1+g(x,y,z)u_2+h(x,y,z)u_3$ does not contain a monomial of the form 
$xy^\sigma$, yielding a contradiction. We conclude that the functions from 
Theorem~\ref{thm_main_v2} are $\GL(3,q)$-inequivalent to monomials.

\item  Firstly, $F\not\approx_p X^2$ since otherwise the semifield 
constructed by $F$ would be isotopic to the finite field, which is clearly not 
the case (which follows from the calculation of the nuclei in Appendix~\ref{sec_nuclei}). 
So assume $F \approx_p X^{p^k+1}$. 
By~\cite[Theorem 2.6.]{coulterhenderson} and the calculation of the nuclei 
in Appendix~\ref{sec_nuclei}, $F \approx_p X^{p^k+1}$ if and only if the associated 
(pre)-semifields are isotopic. We will thus show that the associated semifields 
are not isotopic. The autotopism group of the twisted field corresponding to 
$X^{p^k+1} \in \SQH(1,p^{3m})$ is known, 
see~\cite[Theorem 1.1.]{biliotti1999collineation}, it is
\[
A_k=\{(a^{p^k+1}x^{p^j},\pm ax^{p^j},\pm ax^{p^j}) \colon a \in \F_{q^{3}}^\times, 0 \leq j \leq 3m-1\}.
\]
We now identify some autotopisms of $\S_F$, the (pre)-semifield associated to $F$. It is easy to see that with
\[
     N_{\alpha}=\diag(\alpha^{\sigma+1},\alpha^{\sigma+1},\alpha^{\sigma+1}) \in \GL(3,q),
\;\; L_{\alpha}=\diag(\alpha,\alpha,\alpha) \in \GL(3,q),
\]
we have $(N_{\alpha},L_{\alpha},L_{\alpha})\in \Autt(\S_F)$ for any $\alpha \in \F_{q}^\times$. 
Let $C$ be the subgroup of $\Autt(\S_F)$ that comprises these autotopisms for all  
$\alpha \in \F_q^\times$, i.e., $C$ is a cyclic group of order $q-1$. Let $r$ be a 
$p$-primitive prime divisor that divides $q-1=p^m-1$ but does not divide $p^i-1$ 
for any $1\leq i<m$. Such an $r$ is guaranteed to exist by Zsygmondy's theorem
~\cite[Chapter IX., Theorem 8.3.]{HuppertII} if $m>2$ and $(p,m)\neq (2,6)$, 
and it is immediate that $r\equiv 1 \pmod m$, in particular $r \geq m+1$. Let $R$ be the unique Sylow $r$-subgroup of $C$. 
It is easy to see that autotopism groups of isotopic semifields are conjugate in 
the general linear group, where the conjugacy is actualized by the isotopism
~\cite[Lemma 5.1.]{golouglu2023exponential}. So assume the autotopism groups are 
conjugate via $\gamma \in \GL(3m,p)^3$. Since $\gcd(r,3m)=1$, the only subgroup of 
$A_k$ that can be conjugate to $C$ in the general linear group is 
\[
H_k=\{(a^{p^k+1}x, ax, ax) \colon a \in T\},
\]
where $T$ is the unique Sylow $r$-subgroup of $\F_q^\times$. But then the second and third components of $\gamma$ are necessarily in the normalizer of the set $\{x \mapsto ax \colon a \in T\}$ in $\GL(3m,p)$. This normalizer is exactly $\GammaL(3,q)$, which follows in the same vein as in~\cite[Lemma 5.7.]{golouglu2023exponential}. It is thus easy to see that $\gamma \in \GammaL(3,q)^3$. But since any field automorphism $\tau$ of $\F_{p^{3m}}$ yields an autotopism $(\tau,\tau,\tau)$ of the twisted field (see the definition of $A_k$ for $a=1$), this means that there has to be isotopism in $\GL(3,q)^3$. This implies $F \approx X^{p^k+1}$, contradicting Part (i). 
%\end{proposition}
\end{enumerate}
\end{proof}

Before proving Parts (iii) and (iv) we stop to explain the semifields
listed in the theorem.
Both the Zha-Kyureghyan-Wang (ZKW)~\cite{ZKW} and the Bierbrauer (BB)~\cite{Bierbrauer10}  
pre-semifields are defined through polarizations of planar functions. It is thus enough to 
show that all ZKW and BB planar functions are $\GL(3,q)$-equivalent to a function in 
Theorem~\ref{thm_main}.  The ZKW and BB planar functions are
defined by the Dembowski-Ostrom binomials
\[
F_{r,\omega}=X^{r+1}-\omega^{q-1}X^{qr+q^2} \in \F_{q^3}[X]
\]
where $\omega$ generates $\F_{q^3}^\times$ and $r=p^k$, $q=p^m$. 
Let $d=\gcd(k,m)$ and $s=p^{d}$, that is, $\F_s=\F_q \cap \F_r$. 

These functions are planar if $m/d$ is odd and either
\begin{enumerate}[(i)]
    \item $m/d + k/d \equiv 0 \pmod 3$ (ZKW case), or
    \item  $q \equiv r \equiv 1 \pmod 3$ (BB case).
\end{enumerate}
For convenience, we  exclude the case where $m/d + k/d \equiv 0 \pmod 3$ 
from the BB case, so that the two cases are disjoint. 
We also exclude the $m \mid k$ case since they correspond to the
Albert monomials (twisted fields) by Theorem \ref{thm_menichetti}.
We now show that both of these families are special subfamilies of the family we 
constructed in Theorem~\ref{thm_main_v2}. This immediately proves Part (iii) of 
Theorem~\ref{thm:inequiv}.

\begin{theorem} \label{thm:ZKWBB}
Let $F_{r,\omega}$ be a ZKW or BB planar function on $\F_{q^3}$ as defined 
above. Then $F_{r,\omega}$ is $\GL(3,q)$-equivalent to a function $F$ on $\F_q^3$ from 
Theorem~\ref{thm_main_v2}. In particular, if $F_{r,\omega}$ is a BB function, then $F_{r,\omega}$ is equivalent to a function $F$ on $\F_q^3$ from 
Theorem~\ref{thm_main_v2} with $b=c=0$.
\end{theorem}
\begin{proof}
We start with the ZKW case. 

We first prove that up to $\GL(3,q)$-equivalence, 
we can choose $\omega$ such that 
$\{1,\omega,\omega^{qr+1}\}$ is an 
$\F_q$-basis of $\F_{q^3}$.
Indeed, maps $X \mapsto \gamma X$ for 
$\gamma \in \F_{q^3}^\times$ and then 
scaling by $1/\gamma^{r+1}$ show that $F_{r,\omega'}$ 
is equivalent to 
\begin{equation} \label{eq:scaling}
    X^{r+1}-\omega'^{q-1}\gamma^{qr+q^2-r-1}X^{qr+q^2}=X^{r+1}-(\omega' \gamma^{r-q^2})^{q-1}X^{qr+q^2}=F_{r,\omega' \gamma^{r-q^2}}.
\end{equation}

Using Lemma~\ref{lem:ZKWlemma}, we can thus assume that, up to $\GL(3,q)$-equivalence, $\{1,\omega,\omega^{qr+1}\}$ is an $\F_q$-basis of $\F_{q^3}$. Pick $a,b,c \in \F_q$ such that $\omega$ is a root of $X^{(qr)^2+qr+1}+c^{r^2}X^{qr+1}+b^rX+a=0$. Clearly, such $a,b,c$ always (uniquely) exist if $\{1,\omega,\omega^{qr+1}\}$ is an $\F_q$-basis of $\F_{q^3}$.

Let $v_1=1$, $v_2 = \omega^{q+q^2r}+c^r\omega^q+b$, $v_3 = \omega^q + c$, so $v_2=\omega^qv_3^{qr}+b$.  Clearly, $\{v_1,v_2,v_3\}$ is an $\F_q$-basis of $\F_{q^3}$ since $\{1,\omega,\omega^{qr+1}\}$ is an $\F_q$-basis of $\F_{q^3}$. Let further $u_1=1-\omega^{q-1}$, $u_2=v_3^r-\omega^{q-1}v_3^{qr}$, $u_3=v_3^{r+1}-\omega^{q-1}v_3^{q^2+qr}$. We show that 
$$F_{r,\omega}(xv_1+yv_2+zv_3) = f(x,y,z)u_1+g(x,y,z)u_2+h(x,y,z)u_3$$
for $f,g,h$ as defined in Theorem~\ref{thm_main_v2}.

To verify this, we check the coefficients of all the monomials $x^{r+1},x^ry,x^rz,y^rx,y^{r+1},y^rz,z^rx,z^ry,z^{r+1}$, leading to the following equations.

\begin{minipage}{.5\linewidth}
  \begin{align}
		u_1&=1-\omega^{q-1}\label{eq:l1}\\
		bu_1+cu_2-u_3 &= v_2-\omega^{q-1}v_2^{q^2}\label{eq:l2}\\
		cu_1&=v_3-\omega^{q-1}v_3^{q^2} \label{eq:l3}\\
		0&=v_2^r-\omega^{q-1}v_2^{qr}\label{eq:l4}\\
		au_2&=v_2^{r+1}-\omega^{q-1}v_2^{qr+q^2}\label{eq:l5}
	\end{align}
\end{minipage}%
\begin{minipage}{.5\linewidth}
    \begin{align}
			au_1&=v_2^rv_3-\omega^{q-1}v_2^{qr}v_3^{q^2} \label{eq:r1}\\
			u_2&=v_3^r-\omega^{q-1}v_3^{qr}\label{eq:r2}\\
			bu_2&=v_2v_3^r-\omega^{q-1}v_2^{q^2}v_3^{qr}\label{eq:r3}\\
			u_3&=v_3^{r+1}-\omega^{q-1}v_3^{qr+q^2}.\label{eq:r4}
	\end{align}
\end{minipage}
Eqs.~\eqref{eq:l1}, \eqref{eq:r2}, \eqref{eq:r4} are immediately satisfied by definition of $u_1,u_2,u_3$.
All other equations are readily checked by the definitions as well under the observation that $v_2^r=\omega^{(qr)^2+qr}+c^{r^2}\omega^{qr}+b^r=-a/\omega$.
For instance the right hand side of Eq.~\eqref{eq:l5} reduces to 
	\[v_2^{r+1}-\omega^{q-1}v_2^{qr+q^2}=-\frac{a}{\omega}v_2+\omega^{q-1}\frac{a}{\omega^q}v_2^{q^2}=-\frac{a}{\omega}(\omega^qv_3^{qr}-\omega v_3^r),\]
which is precisely the left hand side of Eq.~\eqref{eq:l5}. The other equations are verified in the same manner.

Note that $\{u_1,u_2,u_3\}$ has to be a basis of $\F_{q^3}$ over $\F_q$, since otherwise the image set of $F_{r,\omega}$ is contained in a proper subspace of $\F_{s^3}$, which clearly contradicts the planar property. This concludes the ZKW case.

We now prove the BB case. Firstly, observe that for a fixed $r$, the function $F_{r,\omega}$ is equivalent to $F_{r,\omega C}$, where $C$ is a cube in $\F_{q^3}$. This follows from Eq.~\eqref{eq:scaling} and the observation that $\gcd(r-q^2,q^2+q+1)=3$ (see~\cite[Lemma 5]{Bierbrauer10}).
% so every cube in $\F_{q^3}$ can be expressed as $a^{r-q^2}$ for some $a \in \F_{q^3}$.
Since $\omega$ is a non-cube in $\F_{q^3}$, there are thus at most two different inequivalent functions for fixed $q,r$. Since $q \equiv 1 \pmod 3$, there are non-cubes in $\F_q$, so we can up to equivalence assume that $\w^3=:u  \in \F_{q}$ from now on.

We now identify $\F_{q^3}$ with $\F_q^3$ via the $\F_q$-basis $\{1,\omega,\omega^2\}$, and write $X=x+y\omega + z \omega^2 \in \F_{q^3}$. We then further identify $X \in \F_{q^3}$ with the vector $(x,y,z)\in \F_q^3$. Simple calculations (under observation of $\omega^3=u \in \F_q$) yield that under this identification
\begin{align*}
    X^{r+1} = (&x^{r+1}+u^{(2r+1)/3}yz^r+u^{(r+2)/3}y^rz,\\
	           &u^{(2r+1)/3}z^{r+1}+x^ry+u^{(r-1)/3}xy^r,\\
			   &u^{(r-1)/3}y^{r+1}+x^rz+u^{(2r-2)/3}xz^r).
\end{align*}

% where 
% \[A=u^{(2r+1)/3}, B=u^{(r+2)/3}, C=u^{(r-1)/3}.\]
Note that $u^q=u$, so we obtain after further calculations
\begin{align*}
    \omega^{q-1}X^{qr+q^2}=u^{(q-1)/3}X^{qr+q^2}
    =(&u^{(q-1)/3}x^{r+1}+u^{(r+2)/3}y^rz+u^{(2qr+1)/3}yz^r,\\
	  &x^ry + u^{(qr+q-2)/3}xy^r+u^{(2r+q)/3}z^{r+1},\\
	  &u^{(2q-2)/3}x^rz+u^{(qr-1)/3}y^{r+1}+u^{(2r-2)/3}xz^r).
\end{align*}
In total, we then have
\begin{align}
    X^{r+1}-\omega^{q-1}X^{qr+q^2} =& ((1-u^{(q-1)/3})x^{r+1}+(u^{(2r+1)/3}-u^{(2qr+1)/3})yz^r,\nonumber\\
    &(u^{(2r+1)/3}-u^{(2r+q)/3})z^{r+1}+(u^{(r-1)/3}-u^{(qr+q-2)/3})xy^r,\nonumber\\
    &(u^{(r-1)/3}-u^{(qr-1)/3})y^{r+1}+(1-u^{(2q-2)/3})x^rz). \nonumber
\end{align}

Let us consider the equivalent function after scaling every coordinate to isolate coefficients:

\begin{align}
    G(x,y,z)=(&x^{r+1}+u\frac{u^{2(r-1)/3}-u^{2(qr-1)/3}}{1-u^{(q-1)/3}}yz^r, \nonumber \\
	          &u z^{r+1}+\frac{u^{(r-1)/3}-u^{(qr+q-2)/3}}{u^{2(r-1)/3}-u^{(2r+q-3)/3}}xy^r, \nonumber \\
			  &y^{r+1}+\frac{1-u^{2(q-1)/3}}{u^{(r-1)/3}-u^{(qr-1)/3}}x^rz). \label{eq:total2}
\end{align}

We consider the three fractions that appear. Let $\zeta =u^{(q-1)/3}$. Note that $\zeta$ has order $3$ and $r \equiv 1 \pmod 3$. Using this, we simplify $u^{(qr-1)/3}=u^{((q-1)r+r-1)/3}=\zeta u^{(r-1)/3}$. Then 
\[\frac{u^{2(r-1)/3}-u^{2(qr-1)/3}}{1-u^{(q-1)/3}} = u^{2(r-1)/3}\frac{1-\zeta^2}{1-\zeta}=-\zeta^2 u^{2(r-1)/3},\]
using that $(1-\zeta^2)/(1-\zeta)=-\zeta^2$. Similarly, 
\[\frac{u^{(r-1)/3}-u^{(qr+q-2)/3}}{u^{2(r-1)/3}-u^{(2r+q-3)/3}} = \frac{u^{(r-1)/3}(1-\zeta^2)}{u^{2(r-1)/3}(1-\zeta)}=-\frac{\zeta^2}{u^{(r-1)/3}}\]
and 
\[\frac{1-u^{2(q-1)/3}}{u^{(r-1)/3}-u^{(qr-1)/3}}=\frac{1-\zeta^2}{u^{(r-1)/3}(1-\zeta)}=-\frac{\zeta^2}{u^{(r-1)/3}}.\]
We can thus simplify Eq.~\eqref{eq:total2} to:
\[
    G(x,y,z)=(x^{r+1}-u \zeta^2 u^{2(r-1)/3} yz^r,u z^{r+1}-\frac{\zeta^2}{u^{(r-1)/3}}    xy^r,y^{r+1}- \frac{\zeta^2}{u^{(r-1)/3}}    x^rz). \]
We now perform the map $z \mapsto \zeta u^{(r-1)/3} z$ that preserves equivalence and arrive at
\[
    G'(x,y,z)=(x^{r+1}-u^{(r^2+r+1)/3} yz^r,\zeta^2 u^{(r^2+2)/3}z^{r+1}-\frac{\zeta^2}{u^{(r-1)/3}}    xy^r,y^{r+1}-    x^rz). \]
Multiplying the second coordinate with $\zeta u^{(r-1)/3}$, setting $y \mapsto -y$ and renaming $y \leftrightarrow z$ yields finally
\[
    G''(x,y,z)=(x^{r+1}+u^{(r^2+r+1)/3} y^rz, u^{(r^2+r+1)/3}y^{r+1}+    xz^r,z^{r+1}-    x^ry), \]
    which is exactly $F$ from Theorem \ref{thm_main_v2} for $a=u^{(r^2+r+1)/3}$, $\sigma =r$, $b=c=0$. 
\end{proof}

\begin{proof}[Proof of Theorem~\ref{thm:inequiv} (iv)]

We only need to prove that the family of semifields we construct in 
Theorem~\ref{thm_main} contains members not contained in the ZKW, BB semifields, 
as well as the family of twisted fields as these semifields are the only ones
that can have the order $q^3$, see Section~\ref{subs:comm_semifields}. 
We already proved in Part (i) that as long as $m > 2$, the family does not intersect 
with the family of twisted fields.

It is easy to find $p,m$ such that the necessary conditions for the 
ZKW and BB semifields are never satisfied. For instance, 
pick a prime $p \equiv 2 \pmod 3$, and $m = 3$. Then $q = p^m \equiv 2 \pmod 3$, 
in particular there are no BB semifields of order $p^{3m}=p^9$ for these choices.
For the ZKW case, for $m = 3$ and arbitrary $p$ we have $k \in \{1,2,4,5,7,8\}$.
In all cases $d=\gcd(k,m)=1$ and $m/d+k/d \not\equiv 0 \pmod 3$, violating the ZKW 
conditions. We conclude that there are also no proper ZKW semifields of order 
$p^{3m}=p^9$. However, such restrictions do not apply to the semifield 
family we construct in Theorem~\ref{thm_main}. Picking for instance 
$k=1$ (still with $m=3$ and the same $p$) yields semifields of order $p^9$. 

\end{proof}

\begin{remark}
We don't study the question of complete inequivalence in this paper
since Theorem~\ref{thm:inequiv} already shows that our construction 
is more general. In fact, our view of these semifields gives a natural 
explanation for the ZKW and BB semifields --- the original ZKW and BB 
semifields are in essence just the semifields in the family of 
Theorem~\ref{thm_main} that can be realized by Dembowski-Ostrom 
binomials.
\end{remark}

% \section{Conclusion}\label{s:conclusion}

%%%%%%%%%%%%%%%%%%%%%%%%%%%%%%%%%%%%%%%%%%%%%%%%%%%%%%%%%%%%%%%%%%%%%%%%%%%%%%%%%%%%%%

\appendix
\section{Invariants of the new semifields} \label{sec_nuclei}

Associative substructures of a semifield $\S = (\F_p^n,+,\ast)$, 
namely the left, middle and right nuclei, are defined as follows:
\begin{align*}
\mathcal{N}_l(\S) &:= \{ x \in \F_p^n \ : \ (x \ast y) \ast z = x \ast (y \ast z), \ y,z \in \F_p^n \},\\
\mathcal{N}_m(\S) &:= \{ y \in \F_p^n \ : \ (x \ast y) \ast z = x \ast (y \ast z), \ x,z \in \F_p^n \},\\
\mathcal{N}_r(\S) &:= \{ z \in \F_p^n \ : \ (x \ast y) \ast z = x \ast (y \ast z), \ x,y \in \F_p^n \}.
\end{align*}
The \emph{nucleus} of $\S$ is then defined as the intersection 
of left, middle and right nuclei, and the \emph{center}, denoted 
by $\mathcal{C}(\S)$ is comprised of all elements in the nucleus 
that commute with all elements in $\S$. It is easy to check that 
$\mathcal{N}_l(\S),\mathcal{N}_m(\S),\mathcal{N}_r(\S) \subseteq \F_p^n$ 
are finite fields and if $\S$ is commutative then 
$\mathcal{N}_l(\S) = \mathcal{N}_r(\S)$. 
Nuclei and center are isotopy invariants for semifields. They are thus useful in determining if different semifields are isotopic or not.
Since a pre-semifield $\P$ is always isotopic 
to the corresponding semifield constructing 
via Kaplansky's trick, the invariants can be 
extended to pre-semifields as well. 

In this section, we calculate the nuclei and the center of the semifields 
associated to the pre-semifields constructed in Theorem~\ref{thm_main}. 
This will make it easier to decide if new constructions of semifields 
are isotopic to our semifields or not.

Let us start with the precise definition of the pre-semifield multiplication 
induced by the functions in Theorem~\ref{thm_main}.

\begin{definition}\label{def_pre-semifield}
Let $q$ be an odd prime power, $\id \ne \sigma \in \Aut(\F_q)$ has odd order
and $a,b,c \in \F_q$ such that 
$x^{\sigma^2+\sigma+1}+cx^{\sigma^2+\sigma}+bx^{\sigma^2}+a=0$ 
has no solution $x \in \F_q$. 
We then define the pre-semifield $\P_{\sigma,a,b,c}=(\F_{q}^3,+,\ast)$ via the multiplication 
induced by the polarization 
\[
x \ast y = F(x+y) -F(x) -F(y),
\]
where $F$ as in Theorem \ref{thm_main}, that is to say,
\[
\begin{pmatrix} x \\ y \\z \end{pmatrix} \ast \begin{pmatrix} u \\ v \\w \end{pmatrix} = 
\begin{pmatrix} 
x^{\sigma}u+xu^{\sigma}+a(y^{\sigma}w+v^{\sigma}z)+b(x^{\sigma}v+yu^{\sigma})+c(x^{\sigma}w+u^{\sigma}z) \\ 
a(y^{\sigma}v+yv^{\sigma})+z^{\sigma}u+xw^{\sigma}+b(z^{\sigma}v+w^{\sigma}y)+c(x^{\sigma}v+u^{\sigma}y) \\
z^{\sigma}w+zw^{\sigma}-(x^{\sigma}v+u^{\sigma}y) 
\end{pmatrix}.  
\]
\end{definition}

 The nuclei and center of the pre-semifields $\P_{\sigma,a,b,c}$ can be computed using a
 theorem of Marino and Polverino \cite[Theorem 2.2]{MP12}.

Let $\P = (\F_p^n,+,\ast)$ be a commutative pre-semifield with right 
multiplication defined as
\[
	R_U : X \mapsto X \ast U, \quad \textrm{for } U \in \F_p^n.
\]
Then the {spread set} associated to $\P$ is defined as
\[
	\mathcal{R} = \{ R_U : U \in \F_p^n \}.
\]

\begin{theorem}\cite[Theorem 2.2]{MP12}\label{MP}
Let $\mathcal{A},\mathcal{B} \subset \End_{\F_p}(\F_p^n)$ be the largest sets 
(and then necessarily fields) such that
\[
	\mathcal{R} \mathcal{A}\subseteq \mathcal{R} \textrm{ and }  \mathcal{B} \mathcal{R} \subseteq \mathcal{R}.
\]
Then
\[
	\mathcal{N}_m(\P) \cong \mathcal{A}\textrm{ and } \mathcal{N}_l(\P) = \mathcal{N}_r(\P) \cong \mathcal{B}.
\]
The center $\mathcal{C}(\P)$ is isomorphic to the largest field $K_w(\P)\subseteq \mathcal{N}_r(\P)$ such that $C_1C_2=C_2w^{-1}C_1w$ for all $C_1 \in K_w(\P)$, $C_2 \in \mathcal{R}$ for some invertible $w \in \mathcal{R}$. 
\end{theorem}

Now, consider $\P_{\sigma,a,b,c}$ on $\F_q^3$.
Then $\mathcal{R} = \{ R_{u,v,w} : u,v,w \in \F_q \}$, where
\[
	R_{u,v,w} : \begin{pmatrix}x \\ y \\z \end{pmatrix} \mapsto \begin{pmatrix} R^{(1)}_{u,v,w}(x,y,z) \\ R^{(2)}_{u,v,w}(x,y,z) \\R^{(3)}_{u,v,w}(x,y,z) \end{pmatrix},
\]
with
\begin{align*}    	                    
    R^{(1)}_{u,v,w}(x,y,z)&=x^{\sigma}u+xu^{\sigma}+a(y^{\sigma}w+v^{\sigma}z)+b(x^{\sigma}v+u^{\sigma}y)+c(x^{\sigma}w+u^{\sigma}z),\\
     R^{(2)}_{u,v,w}(x,y,z) &=  a(y^{\sigma}v+yv^{\sigma})+z^{\sigma}u+xw^{\sigma}+b(z^{\sigma}v+w^{\sigma}y)+c(x^{\sigma}v+u^{\sigma}y), \\
    R^{(3)}_{u,v,w}(x,y,z) &=z^{\sigma}w+zw^{\sigma}-(x^{\sigma}v+u^{\sigma}y). 
\end{align*}
\begin{remark}
We exclude the case $\sigma = \id$. In that case, $F \in \QH(3,q)$. Thus, 
by Theorem \ref{thm_menichetti}, we have
$F \approx X^2        \in \QH(1,q^3)$ or 
$F \approx X^{q+1}   \in \QH(1,q^3)$. 
\end{remark}

The following theorem determines the nuclei of our semifields.

\begin{theorem}
Let $\P_{\sigma,a,b,c}$ be the pre-semifield of Definition \ref{def_pre-semifield}. 
Let $\Fix_{q}(\sigma) = \F_{p^d}$.
Then
\[
	\mathcal{N}_l(\P)=\mathcal{N}_r(\P)=\mathcal{C}(\P)=\mathcal{N}_m(\P)  \cong \F_{p^d}.
\]
\end{theorem}
\begin{proof}
For this proof, we write  $A \in \End_{\F_p}(\F_q^3)$ as 
\[
A : \begin{pmatrix} x \\ y \\z \end{pmatrix}  \mapsto 
\begin{pmatrix} 
\alpha_1(x)+\alpha_2(y)+\alpha_3(z)\\ 
\beta_1(x)+\beta_2(y)+\beta_3(z)  \\
\gamma_1(x)+\gamma_2(y)+\gamma_3(z)  
\end{pmatrix},
\]
where $\alpha_i,\beta_i\gamma_i\in \End_{\F_p}(\F_q)$. We apply Theorem~\ref{MP}. 
We have $A \in \mathcal{B}$ if and only if for each $(u,v,w) \in \F_q^3$, 
there is a $(r,s,t) \in \F_q^3$ such that 
\begin{equation} \label{eq:nuclei1}
    \alpha_i R^{(1)}_{u,v,w} + \beta_i R^{(2)}_{u,v,w} + \gamma_i R^{(3)}_{u,v,w} = R^{(i)}_{r,s,t}
\end{equation}
holds for all $i \in \{1,2,3\}$. Firstly, note that $u=v=w=0$ clearly 
has to correspond to $r=s=t=0$. We first consider the case $i=3$. 
We compare coefficients of the monomials containing $x$ in 
Eq.~\eqref{eq:nuclei1} for $v=w=0$:
% \[\alpha_3(x)(x^{\sigma}u+xu^{\sigma}+bx^{\sigma}v+cx^{\sigma}w)+\beta_3(xw^{\sigma}+cx^{\sigma}v)+\gamma_3(-x^{\sigma}v)=-x^{\sigma} s.\]
\[
 	\alpha_3(x^{\sigma}u+xu^{\sigma})=-x^{\sigma} s.
\]
Then $\alpha_3 \in \End_{\F_p}(\F_q)$ is necessarily the zero polynomial. Indeed, 
by writing $\alpha_3=\sum_{i=0}^{m-1}c_ix^{p^i}$ as a linearized polynomial and 
comparing coefficients again, we get (recalling $\sigma=p^k$),
\[\alpha_3(x^{\sigma}u+xu^{\sigma})=\sum_{i=0}^{m-1}c_i(x^{p^{i+k}}u^{p^i}+x^{p^i}u^{p^{i+k}})=\sum_{i=0}^{m-1}\left(c_{i-k}u^{p^{i-k}}+c_{i}u^{p^{i+k}}\right)x^{p^i},\]
where indices are taken modulo $m$. For all $u \in \F_q$ this has to evaluate as a polynomial to $-x^{\sigma} s$ for some $s \in \F_q$. This means for all $i \neq k$, we have $c_{i-k}u^{p^{i-k}}+c_{i}u^{p^{i+k}}=0$ which implies $c_{i-k}=u^{p^{2k}}c_{i}$. Since the coefficients $c_i$ are independent of $u$, this can hold for all $u \in \F_q$ only if $c_{i-k}=c_{i}=0$ for all $i\neq k$. Since $k\neq 2m$ (otherwise $m/\gcd(k,m)$ is even, violating Theorem~\ref{thm_main}), we conclude $c_i=0$ for all $i$, and $\alpha_3$ is the zero polynomial as claimed.
Checking similarly the coefficients 
of the monomials containing $y$ for $u=w=0$ yields:
\[
\beta_3(a(y^{\sigma}v+yv^{\sigma}))=-r^{\sigma}y,
\]
implying $\beta_3=0$ using a similar argumentation. So 
Eq.~\eqref{eq:nuclei1} for $i=3$ reduces to:
\[
\gamma_3(z^{\sigma}w+zw^{\sigma}-(x^{\sigma}v+u^{\sigma}y))=
         z^{\sigma}t+zt^{\sigma}-(x^{\sigma}s+r^{\sigma}y).
\]
So $\gamma_3(X)=C\cdot X$ necessarily for $C \in \F_q^*$ and $Cv=s$, 
$Cw=t$, $Cw^{\sigma}=t^{\sigma}$, $Cu^{\sigma}=r^{\sigma}$ 
This implies in particular that $C \in \F_{p^d}$. We conclude that any
$A \in \mathcal{B}$ is necessarily of the form
\begin{equation} \label{eq:nuclei1specific}
 A : \begin{pmatrix} x \\ y \\z \end{pmatrix}  \mapsto 
\begin{pmatrix} 
\alpha_1(x)+\alpha_2(y) \\ \beta_1(x)+\beta_2(y)  \\\gamma_1(x)+\gamma_2(y)+Cz  \end{pmatrix} 
\end{equation}
for some $C \in \F_{p^d}$. By inspection, it is easy to see that 
\[A_C : \begin{pmatrix} x \\ y \\z \end{pmatrix}  \mapsto \begin{pmatrix} Cx \\ Cy  \\Cz  \end{pmatrix}  \in \mathcal{B}\]
for $Cu=r, Cv=s, Cw=t$. Let $A \in \mathcal{B}$ be of the form in Eq.~\eqref{eq:nuclei1specific}. Then $A-A_C \in \mathcal{B}$ since $\mathcal{B}$ is as a field closed under subtraction. But $A-A_C$ is not invertible since it does not depend on $z$, so $A=A_C$, again because $\mathcal{A}$ is a field. So $\mathcal{B}$ consists exactly of the $A_C$ for $C \in \F_{p^d}$. In particular, $\mathcal{B} \cong \F_{p^d}$.

To determine the center, it is easy to see that all elements in $\mathcal{B}$ commute with the elements in $\mathcal{R}$, so $\mathcal{C}(\P)\cong \mathcal{B}$ as well.

% Also note that $u,v,w$ vanish if and only if $r,s,t$, respectively, vanish.

% We can proceed similarly for $i=1,2$, we only outline the $i=1$ case. Checking the monomials in $z$ of Eq.~\eqref{eq:nuclei1} for $i=1$ for $u=v=0$ yields (recall $u=v=0$ implies $r=s=0$)
% \[\gamma_1(z^{\sigma}w+z^{\sigma}w)=0,\]
% so $\gamma_1=0$. The same procedure for the monomials in $y$ and $v=w=0=s=t$ yields 
% \[\beta_1(cu^{\sigma}y)=0,\]
% so $\beta_1=0$ as well. So Eq.~\eqref{eq:nuclei1} for $i=1$ becomes
% \begin{align*}
%     \alpha_1(x^{\sigma}u+xu^{\sigma}+a(y^{\sigma}w+v^{\sigma}z)&+b(x^{\sigma}v+u^{\sigma}y)+c(x^{\sigma}w+u^{\sigma}z)) \\&= x^{\sigma}r+xr^{\sigma}+a(y^{\sigma}t+s^{\sigma}z)+b(x^{\sigma}s+r^{\sigma}y)+c(x^{\sigma}t+r^{\sigma}z).
% \end{align*}
% Setting various combinations of $u,v,w$ to zero leads to $\alpha_1(X)=C_1 X$ with $C_1=C_3$. Similarly, by considering $i=2$ in Eq.~\eqref{eq:nuclei1} yields $\alpha_2=\gamma_2=0$ and $\beta_2(X)=C_3 X$. We conclude that 
% $A : \begin{pmatrix} x \\ y \\z \end{pmatrix}  \mapsto \begin{pmatrix} C_3 x \\ C_3y \\C_3z  \end{pmatrix} $, and since $C_3$ is in $K$, we have that $\mathcal{N}_l(\P)=\mathcal{N}_r(\P) \cong \mathcal{B} \cong K$. 

Let us now consider the middle nucleus.
We have $A \in \mathcal{A}$ if and only if for each $(u,v,w)\in \F_q^3$, there is a $(r,s,t) \in \F_q^3$ such that 
\begin{equation} \label{eq:nuclei2}
    R^{(i)}_{u,v,w} \begin{pmatrix} \alpha_1(x)+\alpha_2(y)+\alpha_3(z) \\ \beta_1(x)+\beta_2(y)+\beta_3(z)  \\\gamma_1(x)+\gamma_2(y)+\gamma_3(z)  \end{pmatrix} = R^{(i)}_{r,s,t}\begin{pmatrix} x \\ y \\z \end{pmatrix} 
\end{equation}
for $i \in \{1,2,3\}$.
Firstly, by direct inspection, we can see that for $C \in \F_q \cap \F_{p^{2d}}$ the endomorphisms 
\begin{equation} \label{eq:nucleispecific}
    A_C : \begin{pmatrix} x \\ y \\z \end{pmatrix}  \mapsto \begin{pmatrix} Cx \\ Cy  \\Cz  \end{pmatrix}  \in \mathcal{A}
\end{equation}
 since they satisfy Eq.~\eqref{eq:nuclei2} for $r=uC^\sigma$, $s=vC^\sigma$, $t=wC^\sigma$. We now show that these are the only elements in  $\mathcal{A}$. 

We consider Eq.~\eqref{eq:nuclei2} for $i=3$. 

% Isolating the monomials in $x$ for $u=v=0$, we get
% \[\gamma_1(x)^\sigma w +\gamma_1(x) w^\sigma = -x^\sigma s.\]
% This has to have a solution for all choices of $w$, which is clearly only possible for $\gamma_1=0$. Similarly, considering the monomials in $y$ for $u=v=0$ yields $\gamma_2=0$ in the same way. 

Isolating the monomials in $z$, we get
\begin{equation} \label{eq:nuclei2_1}
    \gamma_3(z)^{\sigma}w + \gamma_3(z)w^{\sigma}-\alpha_3(z)^{\sigma}v-u^{\sigma} \beta_3(z)=z^\sigma t + zt^\sigma.
\end{equation}

For $v=w=0$, we have
\[
-u^\sigma \beta_3(z)=z^\sigma t + zt^\sigma,
\]
For all choices of $u \in \F_q$, there has to be a $t \in \F_q$ such that this 
polynomial equation in $z$ is satisfied. This is only possible for $\beta_3=0$
which can be rigorously checked again by expanding $\beta_3$ as a linearized 
polynomial and comparing coefficients.
Similarly, the monomials in $z$ for $u=w=0$ give $\alpha_3 = 0$. 
So Eq.~\eqref{eq:nuclei2_1} gives
\[ 
\gamma_3(z)^{\sigma}w + \gamma_3(z)w^{\sigma} = z^\sigma t + zt^\sigma.
\]
This polynomial equation only has one possible solution: $\gamma_3(X)=CX$
for some $C \in \F_q$ since $\sigma^2=\id$ is impossible. Let $\gamma_3(X)=CX$.
The equation holds if $C^{\sigma}w=t$ and $Cw^{\sigma}=t^{\sigma}$, which can hold 
simultaneously only if $C \in \F_{p^{2d}}$. We conclude that 
$\alpha_3=\beta_3=0$ and $\gamma_3(X)=CX$ for $C \in \F_{p^{2d}} \cap \F_q$. 
But note that (see Theorem~\ref{thm_main}), we have for $q=p^m$ that $m/d$ is odd, so $\F_{p^{2d}} \cap \F_q=\F_{p^d}$.

So a mapping $A\in \mathcal{A}$ has necessarily the structure  
\[
A : \begin{pmatrix} x \\ y \\z \end{pmatrix}  \mapsto 
\begin{pmatrix} 
\alpha_1(x)+\alpha_2(y) \\ \beta_1(x)+\beta_2(y)  \\\gamma_1(x)+\gamma_2(y)+Cz \end{pmatrix}  
\in \mathcal{A}
\]
for some $C \in \F_{p^{d}}$.  With the same argumentation as 
for the left/right nuclei see Eq.~\eqref{eq:nuclei1specific} and the following lines, 
we conclude that $\mathcal{A}$ only consists 
of the mappings in Eq.~\eqref{eq:nucleispecific}. In particular, 
$\mathcal{N}_m(\P) \cong \mathcal{A} \cong \F_{p^{d}}$.
% \[\beta_3(a(y^{\sigma}v+yv^{\sigma})+bw^{\sigma}y+cu^{\sigma}y)+\gamma_3(-u^{\sigma}y)=-r^{\sigma}y,\]
\end{proof}

%%%%%%%%%%%%%%%%%%%%%%%%%%%%%%%%%%%%%%%%%%%%%%%%%%%%%%%%%%%%%%%%%%%%%%
\section{Auxiliary lemmas} \label{sec:lemmas}

We need a few auxiliary results on the number of $\F_q$-roots of 
polynomials of the form 
\begin{equation*}
P(x)=a_3x^{\sigma^2+\sigma+1} + a_2x^{\sigma+1} + a_1x + a_0,
\end{equation*}
where $\sigma = p^k$ and $q = p^m$ with $a_3 \ne 0$.
These polynomials are called \emph{projective polynomials} and have been 
studied extensively, see for instance
\cite{abhyankar1997projective,von2010composition,mcguire2019characterization}. 
The number of roots of $P$ is necessarily an element of 
$\{0,1,2,3,p^d+1,p^d+2,p^{2d}+p^d+1\}$ where $d=\gcd(k,m)$ 
\cite[Theorem 11]{mcguire2019characterization}. 
Note that, no matter the choice of $\sigma$ and $q$, there are 
always $a_i$ such that $P$ has no roots in $\F_q$. This is well 
known and essentially follows from the existence of irreducible 
elements in skew-polynomial rings of any degree $d\geq 1$, 
see \cite[Theorem 1]{odoni1999additive}.

\subsection{A lemma on the number of $\F_q$-roots of projective polynomials}
Let $q = p^m$ and $\tau = p^d$ with $d \mid m$. The polynomials of the form
\begin{equation}
\label{eq_linearized}
	B(x) = \sum_{i = 0}^n b_i x^{\tau^i} \in \F_{q}[x].
\end{equation}
are called $\F_\tau$-linearized polynomials over $\F_q$.
They form a non-commu\-ta\-tive ring $\mathfrak{R}_{q,\tau}$
where multiplication is defined as polynomial composition and 
addition is the usual addition of polynomials. 
Ore introduced \textit{skew-polynomial rings} in \cite{Ore1}, 
showed that $\mathfrak{R}_{q,\tau}$ gives an example and proved 
\cite{ore1933special} the following result on when 
$x^\tau-\alpha x \in \mathfrak{R}$
satisfies 
$B = (x^\tau -\alpha x) \circ R$ or $B = R \circ (x^\tau -\alpha x)$ 
for some $\F_\tau$-linearized $R \in \mathfrak{R}_{q,\tau}$. 

\begin{lemma} \cite[Theorem 3 and Eq. (11)]{ore1933special}
\label{lem_ore}
Let $\tau = p^d$ and $q = \tau^{m/d}$. Define,
\[
	Q_\tau(x) = \sum_{i = 0}^{n} b_i x^{(\tau^i-1)/(\tau-1)} \in \F_{q}[x],
\]
and
\[
	\widetilde{Q}_\tau(x) = \sum_{i = 0}^{n} b_{n-i}^{\tau^i} x^{(\tau^i-1)/(\tau-1)} \in \F_{q}[x],
\]
where $b_0 \ne 0$ and $b_n \ne 0$. For $\alpha \in \F_q^\times$, we have,
\begin{enumerate}
\item $B = R \circ (x^\tau -\alpha x)$ for some $\F_\tau$-linearized $R \in \mathfrak{R}$ 
$\iff Q_\tau(\alpha) = 0$, and 
\item $B = (x^\tau -\alpha x) \circ R$ for some $\F_\tau$-linearized $R \in \mathfrak{R}$ 
$\iff \widetilde Q_\tau(\alpha^{-\tau}) = 0$. 
\end{enumerate}
\end{lemma}
Observe that $B(x) = xQ_\tau(x^{\tau-1})$ and 
             $\widetilde B(x) = x\widetilde Q_\tau(x^{\tau-1})$,
where
\[
	\widetilde B(x) = \sum_{i = 0}^n b_{n-i}^{\tau^i} x^{\tau^i} \in \F_{q}[x].
\]
The $\F_{\tau}$-linear trace function $\Tr : \F_{q} \to \F_{\tau}$, 
defined as
\[
	\Tr(x) = \sum_{i=0}^{m/d-1} x^{\tau^{i}},
\]
where $\tau^{-i}$ denotes inverse Frobenius, defines a non-degenerate bilinear form $\Tr(xy)$ with respect to which 
the adjoint $B^*$ of $B$ is found via
\[
	\Tr(B(x)y) = \Tr(xB^*(y)),
\]
when $B$ and $B^*$ are viewed as $\F_\tau$-linear endomorphisms of $\F_q$.
Thus
\[
	B^*(y) = \sum_{i = 0}^n b_i^{\tau^{-i}} y^{\tau^{-i}}, 
\]
since 
\[
	\Tr(x b_i^{\tau^{-i}} y^{\tau^{-i}}) =
	\Tr(x b_i^{\tau^{-i}} y^{\tau^{-i}})^{\tau^i} =
	\Tr(b_i x^{\tau^i} y)
\]
where we allow negative exponents as Ore \cite[p. 572]{ore1933special}. 
Define $\overline{B}$ as
\[
\overline{B}(x) := x^{\tau^n} \circ B^*(x) = \widetilde{B}(x),
\]
which brings $B^*$ back to \eqref{eq_linearized}.
%It is evident that $\overline{\overline{B}} = B$ and $\overline{BC}=\overline{C}\overline{B}$
%since $B^{**} = B$ and $(BC)^{*} = C^*B^*$.
It is evident that 
the polynomials $B$ and $\widetilde B$ have 
the same number of distinct $\F_q$-roots, i.e.,
\[
\dim_\tau \ker B = \dim_\tau \ker B^* = \dim_\tau \ker \widetilde{B} = h,
\]
and 
\begin{equation}\label{eq_ore}
|N_q(Q_\tau) \cap (\F_q^\times)^{(\tau-1)}| = \frac{\tau^h-1}{\tau -1} = |N_q(\widetilde Q_\tau) \cap (\F_q^\times)^{(\tau-1)}|,
\end{equation}
where we denote the set of $\F_q$-roots of a 
polynomial $P \in \F_q[x]$ by
\[
N_q(P) = \{x \in \F_q : P(x) = 0 \}
\]
and $S^{(d)} := \{ s^d : s \in S \}$.

In the following we will consider the case 
$q=p^m$, $\sigma = p^k$ and $d = \gcd(m,k)$. 
Let
\[
	P_\sigma(x) = \sum_{i = 0}^{n} a_i x^{(\sigma^i-1)/(\sigma-1)} 
\quad \textrm{ and } \quad 
	\widetilde{P}_\sigma (x) = \sum_{i = 0}^{n} a_{n-i}^{\sigma^i} 
	x^{(\sigma^i-1)/(\sigma-1)},
\]
where $a_0 \ne 0$ and $a_n \ne 0$. 
We will show that 
\begin{lemma}\label{lem_ore_2}
$|N_q(P_\sigma)| = |N_q(\widetilde{P}_\sigma)|$.
\end{lemma}
That is to say, we need to prove a variation of Ore's lemma 
as expressed in Eq. \eqref{eq_ore}, for the roots of $P_\sigma$
and $\widetilde P_\sigma$, over all cosets of $(\F_q^\times)^{(\sigma-1)}$,
and for general $\sigma = \tau^{k/d}$.
 Recall that $\tau = p^d$
and $\gcd(\sigma-1,q-1) = \gcd(\tau-1,q-1)$.
We first need another lemma.

\begin{lemma} \label{lem_orelemma}
Let $\Xi = \{\xi_0, \ldots, \xi_{\tau-2}\}$ be 
a set of representatives for $\F_q^\times/(\F_q^\times)^{(\tau-1)}$ and define
\begin{align*}
	A_j(x) := x P_\sigma(\xi_j x^{\sigma-1}) 
	        = \sum_{i = 0}^{n} a_i \xi_j^{(\sigma^i-1)/(\sigma-1)}x^{\sigma^i},
\end{align*}
for $0 \le j \le \tau-2$. We have
\[
	|N_{q}(P_\sigma)| = \sum_{j=0}^{\tau-2} \frac{|N_{q}(A_j)| - 1}{\tau- 1}.
\]

\end{lemma}
\begin{proof}
Let $\beta \in \F_q^\times$ be a root of $P_\sigma$. 
Then $\beta=\xi_jx^{\sigma-1}$ has solutions in $\F_q^\times$
if and only if 
$\beta \in \xi_j (\F_q^\times)^{(\tau-1)}$, 
and this happens for exactly one $j$.
Each such solution $x$ is a root of $A_j$.
Moreover, 
%if $\beta$ is a root of $A_j$, 
%then so is $\beta c$ for $c \in \F_\tau^\times$. 
$x \mapsto \xi_j x^{(\sigma-1)}$
on $\F_q^\times$ 
is $(\tau-1)$-to-$1$
(since $x \mapsto x^{(\sigma-1)}$ has kernel $\F_\tau^\times$). 
The only other root of each $A_j$ is $0$ 
(which is not a root of $P_\sigma$). 
This proves the statement.
\end{proof}

Note that $A_j$ is $\F_\tau$-linearized (i.e., of type \eqref{eq_linearized} 
since $\sigma = \tau^{k/d}$). Now we can prove Lemma \ref{lem_ore_2}.

\begin{proof}[Proof of Lemma \ref{lem_ore_2}]
Observe
\begin{align*}
\overline{A_j}(x)  
&= \sum_{i = 0}^{n} a_i    ^{\sigma^{n-i}}\xi_j^{\sigma^{n-i}(\sigma^i    -1)/(\sigma-1)} x^{\sigma^{n-i}}\\
&= \sum_{i = 0}^{n} a_{n-i}^{\sigma^{i}  }\xi_j^{\sigma^{i  }(\sigma^{n-i}-1)/(\sigma-1)} x^{\sigma^{i}}\\
&= \sum_{i = 0}^{n} a_{n-i}^{\sigma^{i}  }\xi_j^{(\sigma^{n}-\sigma^i)/(\sigma-1)} x^{\sigma^{i}}.
\end{align*}
Now
\begin{align*}
	{\xi_j}^{-(\sigma^n-1)/(\sigma-1)} \overline{A_j}(x) 
	&= \sum_{i = 0}^{n} a_{n-i}^{\sigma^i} {\xi_j}^{-(\sigma^i-1)/(\sigma-1)}x^{\sigma^{i}} \\
	&= \widetilde{A}_{j}(x),
\end{align*}
since for $\widetilde{A}_j$, we can select the set of representatives $\widetilde{\Xi}$
in Lemma \ref{lem_orelemma} as $\widetilde{\xi}_{j} = 1/\xi_j$ for $\xi_j \in \Xi$.
We conclude that the number of roots  of $\widetilde{A}_{j}$ in $\F_q$ 
is the same as the number of roots of $A_j$, and the result follows from 
Lemma~\ref{lem_orelemma}. 
\end{proof}

Finally, we will prove the following lemma 
which is actually needed in the proof of Theorem~\ref{thm_main_v2}. 

\begin{lemma} \label{lem:auxiliary}
Let $a,b,c \in \F_q$ with $a\neq 0$. Then
\[
\Pi_1(x) = x^{\sigma^2+\sigma+1}-bx^{\sigma^2+\sigma}+ac^{\sigma}x^{\sigma^2}-a^{\sigma+1} \in \F_q[x]
\]
has the same amount of roots in $\F_q$ as 
\[
\Pi_2(x) = x^{\sigma^2+\sigma+1} + cx^{\sigma^2+\sigma} + bx^{\sigma^2} + a\in \F_q[x].
\]
\end{lemma}
\begin{proof}

We first perform a series of transformations to $\Pi_1$ that 
leave the number of roots invariant. First we apply 
$\sigma$ to every coefficient, which leaves 
the number of roots invariant (i.e., $x^\sigma \circ \Pi_1 \circ x^{\sigma^{-1}}$). 
Then we apply $x \mapsto -xa$ and divide the resulting polynomial 
by $a^{\sigma^2+\sigma}$. We arrive at the polynomial
\[
ax^{\sigma^2+\sigma+1}+b^\sigma x^{\sigma^2+\sigma}+c^{\sigma^2}x^{\sigma^2}+1.
\]
Reciprocation (the transformation $x \mapsto 1/x$ and multiplying the 
result by $x^{\sigma^2+\sigma+1}$) yields
\[
x^{\sigma^2+\sigma+1}+c^{\sigma^2}x^{\sigma+1}+b^\sigma x+a.
\]

By Lemma~\ref{lem_ore_2}, this polynomial has the same number of 
roots as $ax^{\sigma^2+\sigma+1} + bx^{\sigma+1} + cx + 1$.
Now by another reciprocation we arrive at
\[
x^{\sigma^2+\sigma+1} + cx^{\sigma^2+\sigma}+bx^{\sigma^2}+ a = \Pi_2(x).
\]
\end{proof}

\subsection{A generalization of a theorem of Weng and Zeng}

We need the following simple generalization of 
\cite[Proposition 3.6.]{weng2012further}
in the proof of Theorem \ref{thm_sqh}.

\begin{proposition} \label{prop_wengzeng}
Let $q=p^m$ be odd, and $F \in \SQH(n,q)$ satisfy
\begin{enumerate}
\item $F(\U x) = 0 \iff \U x = \U 0$ and
\item $F$ is $2$-to-$1$ on $\F_q^n\setminus\{\U{0}\}$. 
\end{enumerate}
Let $\alpha \in \F_q^\times$. Then 
\begin{enumerate}
\item if $n$ is even or  $\alpha$ is a square in $\F_q^\times$ then 
$\alpha\im(F)=\im(F)$, and
\item if $n$ is odd and $\alpha$ is a non-square in $\F_q^\times$ then 
$\alpha\im(F)\cap \im(F)=\{0\}$.
\end{enumerate}
\end{proposition}
\begin{proof}

Let $F \in \SQH(n,q)$ with companion automorphism $\sigma=p^k$ 
and let $\F_{p^d}$ be the fixed field of $\sigma$. 
Since $F$ is $2$-to-$1$, we have by Lemma~\ref{lem_gcd} that 
$m/\gcd(k,m)=m/d$ is odd. Note first that, 
\begin{equation}\label{eq_sq}
\alpha \textrm{ is a square in } \F_q^\times \implies \alpha\im(F)=\im(F)
\end{equation}
since we can write $\alpha=\mu^{\sigma+1}$ for some $\mu \in \F_q^\times$
since $\gcd(\sigma+1,q-1) = 2$ by Lemma \ref{lem_gcd},
so $F(\mu x) = \alpha F(x)$ for all $x \in \F_{q^n}$.
Set $m=dm'$, so $F \in \QH(nm',p^d)$. 
Let $\epsilon \in \F_{p^d}^\times$. 
By~\cite[Proposition 3.6.]{weng2012further}, we have 
\begin{equation}\label{eq_wengzeng}
\epsilon \im(F)= \begin{cases} 
\im(F), & nm' \text{ is even or } \epsilon \text{ is a square in }\F_{p^d} \\
\F_{q}^n\setminus (\im(F)\setminus\{0\}), & nm' \text{ is odd and } 
\epsilon \text{ is a non-square in }\F_{p^d}.
\end{cases}
\end{equation}
The square-roots of non-squares of $\F_{p^d}^\times$ lie in the
degree-$2$ extension $\F_{p^{2d}}$ and $\F_{q} \cap \F_{p^{2d}} = \F_{p^d}$ 
since $m'$ is odd. Therefore $\epsilon \in \F_{p^d}^\times$ is 
a square in $\F_{p^d}^\times$ if and only if $\epsilon$ is a square in $\F_q^\times$,
and in turn, we can write any $\alpha \in \F_q^\times$ as $\alpha=\epsilon\lambda$, 
where $\epsilon\in \F_{p^d}^\times$ and 
$\lambda$ is a square in $\F_q^\times$.
Then 
\[
\alpha\im(F) = \epsilon (\lambda\im(F))=\epsilon  \im(F) =
\begin{cases} 
\im(F),                                   & n \text{ is even or } \epsilon \text{ is a square in }\F_{p^d} \\
\F_{q^n}\setminus (\im(F)\setminus\{0\}), & n \text{ is odd and } \epsilon \text{ is a non-square in }\F_{p^d},
\end{cases}
\]
where we use the fact that $\lambda\im(F)=\im(F)$ by Eq. \eqref{eq_sq}, 
the fact that $m'$ is odd, and Eq.~\eqref{eq_wengzeng}.
\end{proof}

\subsection{A technical lemma for the ZKW case}

\begin{lemma} \label{lem:ZKWlemma}
Let $r=p^k$, $q=p^m$, $d=\gcd(k,m)$ with $m \nmid k$, $m/d$ odd and $m/d + k/d \equiv 0 \pmod 3$. Further, let $H=(\F_{q^3}^\times)^{r-q^2}$ be the group of $(r-q^2)$-th powers. 
For every fixed generator $\omega'$ of $\F_{q^3}^\times$, 
there exists $\omega \in \omega'H$ such that 
$\{1,\omega,\omega^{qr+1}\}$ is a linearly independent set over $\F_q$.
\end{lemma}

\begin{proof}
Let 
\begin{align*}
D(x)&:=\det\begin{pmatrix} 
1 & x & x^{qr+1} \\ 
1 & x^q & x^{q^2r+q} \\ 
1 & x^{q^2} & x^{r+q^2} 
\end{pmatrix} \\
&= x^{q^2+q+r}+x^{q^2r+q+1}+x^{qr+q^2+1}-x^{q^2r+q^2+q}-x^{q^2+r+1}-x^{qr+q+1}\\
&= \Tr_{q^3/q} (x^{q^2+q+r}  - x^{q^2+r+1})\\
&= \Tr_{q^3/q} (\Nm_{q^3/q}(x) (x^{r-1} - x^{r-q}))\\
&= \Nm_{q^3/q}(x) \Tr_{q^3/q} (x^{r-1} - x^{r-q} ),
\end{align*} 
where $\Tr_{q^3/q}(x)=x+x^q+x^{q^2}$ and 
$\Nm_{q^3/q}(x)=x\cdot x^q \cdot x^{q^2}$ are 
the trace and norm maps, mapping $\F_{q^3}$ to $\F_q$.
Recall the fact that the trace map is $\F_q$-linear,
which was used in the last line.
If $\{1,\omega,\omega^{qr+1}\}$ is a linearly dependent set over $\F_q$
then $D(\omega)=0$ since a nontrivial linear combination 
$a+b\omega+c\omega^{qr+1}=0$ with $a,b,c \in \F_q$ implies that $(a,b,c)^T$ 
is in the kernel of the matrix above with $x=\omega$ observing that the rows 
are $q^i$-th powers of the first row for 
$i \in \{0,1,2\}$.

Assume to the contrary that $\{1,\omega,\omega^{qr+1}\}$ is linearly dependent
for every $\w \in \w'H$.
We must have $D(\omega' h)=0$ for all $h \in H = (\F_{q^3}^\times)^{qr-1}$ 
since $\gcd(qr-1,q^3-1)=\gcd(r-q^2,q^3-1)$.
This is equivalent to saying $D(\omega' x^{qr-1}) = 0$ for all $x \in \F_{q^3}$.
Thus, as a polynomial
\[
\Tr_{q^3/q} \left(\omega'^{r-1}x^{(qr-1)(r-1)}- \omega'^{r-q}x^{(qr-1)(r-q)} \right) 
\equiv 0 \pmod {x^{q^3}-x},
\]
A polynomial of the above form satisfies the above condition, i.e., 
$\Tr_{q^3/q}(ax^d-bx^e) \equiv 0 \pmod {x^{q^3}-x}$
only if $x^{dq^i} = x^{eq^j}$  or $x^{dq^i} = x^{dq^j}$
and in these cases one must also have $a^{q^i} = b^{q^j}$
or $a^{q^i} = a^{q^j}$ respectively.
Since $\omega'$ is primitive, we have 
${(\omega'^{r-1})}^{q^i}\neq \omega'^{r-q}$ for $i \in \{0,1,2\}$, 
since 
\begin{align*}
(r-1)q^i \equiv rq^i - q^i \equiv r-q \pmod{q^3-1}
\end{align*}
which is equivalent to 
\begin{align*}
 rq^i + q  \equiv r + q^i \pmod{q^3-1}
\end{align*}
is impossible for $i = 0$ or $i = 1$, 
and leads to contradiction if $i = 2$.
So the polynomial can only vanish if both 
\[
(qr-1)(r-1)\equiv (qr-1)(r-1)q \pmod{q^3-1}
\]
and 
\[
(qr-1)(r-q)\equiv (qr-1)(r-q)q \pmod{q^3-1}.
\]
Expanding the first condition yields 
$qr^2+q^2r+1 \equiv q^2r^2+r+q \pmod{q^3-1}$. 
Thus since $r \ne 1$ and $q \ne 1$ then 
$q^2r^2 \equiv 1 \pmod {q^3-1}$ implying
either $r = q^2$ or $r^2 = q$. 
If $r^2=q$, then $2k=m$, so $d=k$ and $m/d=2$, violating the condition that $m/d$ is odd, so this case does not occur. We have thus proven the claim for all cases except if $r=q^2$. But this contradicts $m \nmid k$. 
\end{proof}

\section*{Acknowledgments}

We thank Chin Hei Chan for reading and spotting mistake in a preliminary version of this paper; and Bill Kantor for pointing out several corrections. We also thank an anonymous reviewer for their suggestions that improved the presentation of the results. F.G.: conceptualization, methodology, investigation (lead), writing—original
draft (equal). L.K.: investigation (supporting), writing—original draft (equal).

\bibliographystyle{amsplain}
\bibliography{semifields} 

\bigskip
\hrule
\bigskip
\end{document}